\documentclass[11pt,a4paper]{amsart}

\usepackage{
    amsfonts, amsmath, amssymb, ascmac,
    comment,
    empheq, enumitem,
    hyperref,
    mathrsfs, mathtools, mleftright% ,
}

\setlist[enumerate]{%
    topsep=2pt,
    leftmargin=6ex
}
\setlist[itemize]{%
    topsep=2pt,
    leftmargin=4ex
}

\allowdisplaybreaks

\makeatletter
\def\@lbibitem[#1]#2{%
  \item[\@biblabel{#1}]%
  \if@filesw
    {\let\protect\noexpand
     \immediate\write\@auxout{\string\bibcite{#2}{#1}}}%
  \fi
  \ignorespaces}
\makeatother

\newtheorem{thm}{Theorem}[section]
\newtheorem{cor}[thm]{Corollary}
\newtheorem{lem}[thm]{Lemma}
\newtheorem{prop}[thm]{Proposition}

\theoremstyle{definition}
\newtheorem{defn}[thm]{Definition}
\newtheorem{exam}[thm]{Example}
\newtheorem*{defn*}{Definition}
\newtheorem*{exam*}{Example}

\theoremstyle{remark}
\newtheorem{rem}[thm]{Remark}

\newcommand{\bigast}{\mathop{\scalebox{1.6}{\raisebox{-0.16ex}{$\ast$}}}}
\newcommand{\bigwtensor}{\mathop{\bar\bigotimes}}

\newcommand{\nS}{\mathcal{S}_{\ast}}
\newcommand{\sa}{\mathrm{sa}}
\newcommand{\wlim}{\text{weak*-}\lim}

\newcommand{\wtensor}{\mathbin{\bar\otimes}}
\newcommand{\C}{\ifmmode C^{\ast}\else C${}^{\ast}$\fi}
\newcommand{\W}{\ifmmode W^{\ast}\else W${}^{\ast}$\fi}
\newcommand{\I}{\mathrm{I}}
\newcommand{\II}{\mathrm{II}}
\newcommand{\III}{\mathrm{III}}

\renewcommand{\hat}{\widehat}
\renewcommand{\mod}{\operatorname{mod}}
\renewcommand{\tilde}{\widetilde}

\DeclareMathOperator{\Ad}{Ad}

\DeclareMathOperator{\Aut}{Aut}

\DeclareMathOperator{\clco}{\overline{co}}
\DeclareMathOperator{\clInt}{\overline{Int}}

\DeclareMathOperator{\Int}{Int}
\DeclareMathOperator{\id}{id}

\DeclareMathOperator*{\esssup}{ess~sup}

\newcommand{\N}{\mathcal{N}}
\newcommand{\U}{\mathcal{U}}
\newcommand{\Z}{\mathcal{Z}}
\newcommand{\VC}{\mathcal{VC}_{\mathrm{lc}}}
\DeclareMathOperator{\SL}{SL}

\title{Group actions on von Neumann algebras with compact open subgroups}
\author{Takumi Nishihara}
\address{RIMS, Kyoto University, 606-8502 Japan}
\email{nishihar@kurims.kyoto-u.ac.jp}
% 46L10: General theory of von Neumann algebras, 46L40: Automorphisms of selfadjoint operator algebras, 46L55: Noncommutative dynamical systems
\date{September 22, 2026}

\begin{document}

\begin{abstract}
    We study strictly outer actions of locally compact groups with a compact open subgroup on von Neumann factors.
    For amenable groups, we prove 2-cohomology vanishing and obtain classification results using a description of the central sequence algebra and Rohlin-type observations.
    We also characterize the inclusions of factors associated with group actions, and extend M.~Choda's result to this locally compact setting.
\end{abstract}

\maketitle

\renewcommand{\thethm}{\Alph{thm}}
\section*{Introduction}

In this paper, we are interested in actions of a second-countable locally compact group $G$ admitting a compact open subgroup $K<G$ on a factor.

Group actions constitute one of the fundamental subjects in the study of operator algebras.
Historically, the classification theory for actions of discrete amenable groups on the hyperfinite $\II_1$ factor $R$ has developed through the work of Connes on automorphisms \cite{Co75, Co77}, Jones on finite group actions \cite{Jo80}, and Ocneanu on actions of countable discrete amenable groups \cite{Oc85}.
A key feature of this theory is the use of asymptotic properties encoded in central sequence algebras.
In particular, the strong outerness condition, called central freeness, plays a crucial role.
Motivated by this, we use the following analogue for continuous group actions.

\begin{defn*}
    An action $\alpha\colon G\curvearrowright M$ is called \emph{centrally free} if the induced action $G\curvearrowright M_{\omega,\alpha}$ on the equicontinuous part of the central sequence algebra is faithful.
\end{defn*}

In the discrete setting, the noncommutative Rohlin lemma is a basic tool.
Amenability provides the underlying F\o lner sets, and local quantization supplies the analytic input needed to construct the associated towers.
These towers exhibit an approximate shift structure for the induced action on the central sequence algebra, and by a Shapiro-type argument, this structure leads to cohomology vanishing results in sequence algebras.
In addition, careful estimates show that analogous cohomology vanishing results hold on the original algebras, and the Rohlin towers are also used in the intertwining arguments for classification of actions.

For nondiscrete groups, one of the most extensively studied classes is flows, namely $\mathbb{R}$-actions.
Kishimoto \cite{Ki96} has introduced the Rohlin property for flows on \C-algebras, and Kawamuro \cite{Ka00} on the hyperfinite $\II_1$ factor.
Later, Masuda and Tomatsu \cite{MT16} have classified Rohlin flows on general von Neumann algebras.
They have further shown that central freeness is equivalent to the Rohlin property for flows on factors (\cite[Theorem 7.8]{MT16}).
More recently, the author has established an analogous framework for compact group actions in \cite{Ni26}.
These developments provide a unified viewpoint on compact and discrete group actions through their behavior on central sequence algebras.
Indeed, when a topological group $G$ has a compact open subgroup $K$, the following two ingredients are available:
\begin{itemize}
    \item The Rohlin property of compact group actions is applied to the restriction to $K$; in particular, Proposition \ref{prop:cpt-cent.free->Rohlin} gives a $K$-equivariant embedding $L^{\infty}(K)\hookrightarrow M_{\omega,\alpha}$; and
    \item The discrete $G$-space $G/K$ admits F\o lner-type approximations when $G$ is amenable.
\end{itemize}
Thus, the compact open subgroup setting combines a compact group direction, controlled by $K$, with a discrete homogeneous direction, controlled by $G/K$.

A basic model case is that of strictly outer actions on the hyperfinite $\II_1$ factor $R$.
We show that every strictly outer action of a locally compact group with a compact open subgroup on $R$ is centrally free (Proposition \ref{prop:str.outer->cent.free}).
The same conclusion holds for strictly outer cocycle actions.
This connects the standard outerness hypothesis with the central sequence formulation needed in the subsequent arguments.

Cocycle actions and cocycle crossed products arise naturally as generalizations of genuine actions and ordinary crossed products.
A basic question is whether a cocycle action can be perturbed into a genuine action.
Sutherland \cite{Su80} has shown that every 2-cocycle on a properly infinite von Neumann algebra is a coboundary, so the essential difficulty lies in the finite algebra case.
For countable discrete amenable groups, Ocneanu \cite{Oc85} has proved 2-cohomology vanishing on the hyperfinite $\II_1$ factor, and Popa \cite{Po21} has generalized this to free cocycle actions on arbitrary $\II_1$ factors.
Analogous results for continuous groups have been obtained in \cite{MT16, Ni26}.
Building on these ideas, we prove that centrally free cocycle actions of amenable locally compact groups with compact open subgroups have vanishing 2-cohomology.
Moreover, Popa's observation concerning amalgamations extends to this setting.

\begin{thm}\label{main:2-coh-vanish}
    Let $\VC$ denote the class of all second-countable locally compact groups such that every centrally free cocycle action on every factor has vanishing 2-cohomology.
    Then $\VC$ contains every amenable locally compact group with a compact open subgroup, and $\VC$ is closed under countable amalgamated free products over a common compact open subgroup.
\end{thm}

Therefore, the class $\VC$ also contains certain groups acting on trees (Corollary \ref{cor:on-tree}).
To the best of the author's knowledge, these provide the first nonamenable, nondiscrete examples in $\VC$.

Following the strategy of \cite{Oc85, Ma13, MT16}, we also obtain the following classification theorem.

\begin{thm}\label{main:classification}
    Let $G$ be an amenable locally compact group which has a compact open subgroup $K<G$, and let $\alpha,\beta\colon G\curvearrowright M$ be centrally free actions on a factor $M$.
    Suppose that $\beta_g\circ\alpha_g^{-1}\in\clInt(M)$ for all $g\in G$.
    Then $\alpha$ and $\beta$ are strongly cocycle conjugate.
    Moreover, the cocycle implementing the cocycle conjugacy may be chosen to be trivial on $K$.

    In particular, a strictly outer action of $G$ on the hyperfinite $\II_1$ factor $R$ is unique up to cocycle conjugacy.
\end{thm}

Our next interest is inclusions of von Neumann algebras associated with group actions.
A strictly outer (cocycle) action $G\curvearrowright M$ gives rise to a regular irreducible inclusion $(M\subset M\rtimes G)$.
Moreover, the quotient $\N_{M\rtimes G}(M)/\U(M)$ recovers $G$ as a topological group (Lemma \ref{lem:N/U}).
These observations lead to a natural question:
\begin{center}
    How is the group-theoretic structure of $G$ reflected in the crossed product inclusion?
\end{center}
For strictly outer actions of amenable groups on $R$, the preceding uniqueness theorem leads to the following description of the relative automorphism group.

\begin{thm}\label{main:short-exact}
    Let $\alpha\colon G\curvearrowright R$ be a strictly outer action of an amenable locally compact group with a compact open subgroup.
    Then there exists a natural short exact sequence of groups
    \[1\longrightarrow\clInt(R\subset R\rtimes G)\longrightarrow\Aut(R\subset R\rtimes G)\longrightarrow\Aut(G)\longrightarrow 1.\]
\end{thm}

A central theme in the relationship between group structure and von Neumann algebra inclusions is the Galois correspondence.
This goes back to H.~Choda \cite{Ch78}, who has treated intermediate subfactors admitting conditional expectations for outer actions of countable discrete groups.
Izumi, Longo, and Popa \cite{ILP98} later have obtained the full Galois correspondence in the discrete setting.
Boutonnet and Brothier \cite{BB18} have established a locally compact analogue for actions of totally disconnected groups under suitable outerness hypotheses, which Lemma \ref{lem:outerness} shows to be automatic for strictly outer actions.
In all these settings, the position of an intermediate subfactor inside the ambient inclusion determines its crossed product decomposition.
Besides, M.~Choda \cite{Ch79} has characterized regular irreducible inclusions with a faithful normal conditional expectation as cocycle crossed products by countable discrete groups.
We obtain the following extension of this characterization to locally compact groups with compact open subgroups.

\begin{thm}\label{main:factrization}
    Let $G$ be a locally compact group with a compact open subgroup $K$.
    Then an inclusion of von Neumann algebras $(M\subset M_1)$ is isomorphic to $(M\subset M\rtimes G)$ for some strictly outer cocycle action $G\curvearrowright M$ if and only if the following conditions hold.
    \begin{enumerate}
        \item The inclusion $(M\subset M_1)$ is regular and irreducible.
        \item There exists a topological group isomorphism $\N_{M_1}(M)/\U(M)\cong G$.
        \item There exists a faithful normal conditional expectation from $M_1$ onto $q^{-1}(K)''$, where $q\colon\N_{M_1}(M)\to\N_{M_1}(M)/\U(M)=G$ denotes the canonical quotient.
    \end{enumerate}
    In particular, if $G=K$ is compact, then the condition \textup{(3)} is automatic.
\end{thm}

When $M=R$ and $G$ is amenable, Theorems \ref{main:2-coh-vanish} and \ref{main:classification} further show that the resulting crossed product inclusion depends only on $G$ up to isomorphism, giving a complete characterization of the crossed product inclusion in this case.

This paper is organized as follows.
Section \ref{sec:1} fixes notation and collects the necessary preliminaries.
In Section \ref{sec:2}, we study centrally free actions, and compare central freeness with other outerness conditions.
We also construct Rohlin towers for amenable group actions.
Section \ref{sec:3} establishes 1- and 2-cohomology vanishing results, containing the proof of Theorem \ref{main:2-coh-vanish}.
In Section \ref{sec:4}, we prove Theorem \ref{main:classification} and discuss further applications concerning McDuff factors.
Finally, Section \ref{sec:5} addresses crossed product inclusions and concludes Theorems \ref{main:short-exact} and \ref{main:factrization}.

{
\def\addcontentsline#1#2#3{}\relax
\subsection*{Acknowledgements}
% The author is thankful to % Narutaka Ozawa and Yusuke Isono for their invaluable comments and support.
The author is supported by JST SPRING, Grant Number JPMJSP2110.

\subsection*{AI statement}
ChatGPT 5.6 Sol assisted with exploratory mathematical discussions, including the development of Proposition \ref{prop:cpt-cent.free->Rohlin}, and with language editing.
The author independently verified all mathematical arguments and takes full responsibility for the content of the paper.
}

\tableofcontents

\renewcommand{\thethm}{\thesection.\arabic{thm}}
\section{Preliminaries}\label{sec:1}

\subsection{General notations}

Throughout this paper, $G$ denotes a second-countable locally compact group with a fixed compact open subgroup $K<G$.
Every von Neumann algebra is assumed to have separable predual unless it is ultraproduct.

The closed unit ball of a Banach space $B$ is denoted by $(B)_1$.
For a von Neumann algebra $M$, we denote its center, unitary group, space of normal states, automorphism group, and inner automorphism group by $\Z(M)$, $\U(M)$, $\nS(M)$, $\Aut(M)$, and $\Int(M)$, respectively.
The predual $M_{\ast}$ carries the canonical $M$-bimodule structure given by $(x\varphi y)(z)=\varphi(yzx)$ for $x,y,z\in M$ and $\varphi\in M_{\ast}$.
For $x\in M$ and $\varphi\in M_{\ast}$, we write $[x,\varphi]=x\varphi-\varphi x\in M_{\ast}$.
The group $\Aut(M)$ acts isometrically on the predual $M_{\ast}$ via $\alpha(\varphi)=\varphi\circ\alpha^{-1}$ for $\alpha\in\Aut(M)$ and $\varphi\in M_{\ast}$.
We equip $\Aut(M)$ with the $u$-topology:
A net $\alpha_i$ in $\Aut(M)$ converges to $\alpha$ if and only if $\|\alpha_i(\varphi)-\alpha(\varphi)\|\to0$ for every $\varphi\in M_{\ast}$, or equivalently, for every $\varphi\in\nS(M)$.
Note that if a sequence $(\alpha_n)_n$ in $\Aut(M)$ satisfies that $(\alpha_n(\varphi))_n$ and $(\alpha_n^{-1}(\varphi))_n$ are both Cauchy sequences in $M_{\ast}$, then $(\alpha_n)_n$ converges in the $u$-topology to an automorphism $\alpha\in\Aut(M)$.
Since the action $\Aut(M)\curvearrowright M_{\ast}$ is isometric, convergence is uniform on compact subsets of $M_{\ast}$.
With the $u$-topology, $\Aut(M)$ is a Polish group whenever $M$ has separable predual.
The approximately inner automorphism group $\clInt(M)$ is the closure of $\Int(M)$ in $\Aut(M)$.
We denote the hyperfinite $\II_1$ factor by $R$ and the separable type $\I_{\infty}$ factor by $\mathbb{B}=\mathbb{B}(\ell^2)$.
The symbol $\wtensor$ denotes the von Neumann tensor product.

We use $u\colon G\times G\to A$, $v\colon G\to A$ and $w\in A$ as generic notation for a 2-cochain, a 1-cochain and an element (0-cochain), respectively.
The evaluation value of a map $\alpha\colon X\to\Aut(M)$ (or an $n$-cochain $c$) may be denoted by $\alpha_x$ instead of $\alpha(x)$ (resp.\ $c_{g_1,\ldots,g_n}$).
Given $\alpha\colon X\to\Aut(M)$, $v\colon X\to\U(M)$ and $\theta\in\Aut(M)$, we will suppress pointwise notation in compositions; for instance, we write $\theta\circ\Ad{v}\circ\alpha\circ\theta^{-1}$ for the map $X\ni x\mapsto \theta\circ\Ad{v_x}\circ\alpha_x\circ\theta^{-1}$.

For a normal state $\varphi$ on $M$, we consider the following seminorms on $M$ defined by
\begin{align*}
    \|x\|_{\varphi} &= \varphi(x^{\ast}x)^{1/2},& \|x\|_{\varphi}^{\sharp} &= \varphi\left(\frac{x^{\ast}x+xx^{\ast}}{2}\right)^{1/2},\\
    |x|_{\varphi} &= \|x\varphi\|_{M_{\ast}}, & |x|_{\varphi}^{\sharp} &= \frac{\|x\varphi\|+\|\varphi x\|}{2}.
\end{align*}
They are norms when $\varphi$ is faithful.
If $\tau$ is a distinguished trace on $M$, we write $\|\cdot\|_2$ and $|\cdot|_1$ instead of $\|\cdot\|_{\tau}=\|\cdot\|_{\tau}^{\sharp}$ and $|\cdot|_{\tau}=|\cdot|_{\tau}^{\sharp}=\tau(|\cdot|)$, respectively.
For every $x\in M$, we have $\varphi(|x|)\le|x|_{\varphi}\le\|x\|_{\varphi}$ and $\|x\|_{\varphi}^2\le\|x\|\varphi(|x|)$.
Consequently, when $\varphi$ is faithful, norms $\|\cdot\|_{\varphi}$ and $|\cdot|_{\varphi}$ (resp.\ $\|\cdot\|_{\varphi}^{\sharp}$ and $|\cdot|_{\varphi}^{\sharp}$) induce the ultrastrong topology (resp. $*$-ultrastrong topology) on bounded subsets of $M$.
Two commuting subalgebras $A,B\subset M$ are said to be \emph{$\varphi$-orthogonal} if $\varphi(ab)=\varphi(a)\varphi(b)$ for all $a\in A$ and $b\in B$.
In this case, canonically $A\vee B=A\wtensor B\subset M$ and $\varphi|_{A\wtensor B}=\varphi|_A\otimes\varphi|_B$.
When the faithful state $\varphi$ is understood from the context, we simply say that $A$ and $B$ are orthogonal.

\begin{lem}\label{lem:1-norm}
    \begin{enumerate}
        \item Let $A\subset B$ be von Neumann algebras, and $\varphi$ be a normal state on $B$.
        Suppose that there exists a $\varphi$-preserving normal conditional expectation $\mathbb{E}\colon B\to A$.
        Then $|a|_{\varphi|_A}=|a|_{\varphi}$ for $a\in A$.
        \item Let $(X,\mu)$ be a probability space, $M$ a von Neumann algebra, and $\varphi$ a normal state on $M$.
        Then
        \[|f|_{\mu\otimes\varphi}=\int_X|f(x)|_{\varphi}\,d\mu(x)\]
        for $f\in L^{\infty}(X)\wtensor M$.
    \end{enumerate}
\end{lem}
\begin{proof}
    (1):
    The standard calculation
    \[\|a\varphi\|_{B_{\ast}}=\|(a\varphi|_A)\circ\mathbb{E}\|_{B_{\ast}}\le\|a\varphi|_A\|_{A_{\ast}}\le\|a\varphi\|_{B_{\ast}}\]
    implies that $|a|_{\varphi|_A}=\|a\varphi|_A\|_{A_\ast}=\|a\varphi\|_{B_{\ast}}=|a|_{\varphi}$.

    (2):
    Under the canonical identification $(L^{\infty}(X)\wtensor M)_{\ast}=L^1(X)\hat{\otimes}M_{\ast}=L^1(X,M_{\ast})$, the functional $f\cdot(\mu\otimes\varphi)$ corresponds to the map $x\mapsto f(x)\varphi$.
    Hence
    \[|f|_{\mu\otimes\varphi}=\|f\cdot(\mu\otimes\varphi)\|_{L^1(X)\hat{\otimes}M_{\ast}}=\int_X\|f(x)\varphi\|_{M_{\ast}}\,d\mu(x)=\int_X|f(x)|_{\varphi}\,d\mu(x)\]
    for every $f\in L^{\infty}(X)\wtensor M$.
\end{proof}

We normalize a left Haar measure on the group $G$ so that $|K|=1$.
Let $\pi\colon G\to G/K$ be the canonical quotient map, and choose a section $\sigma\colon G/K\to G$ satisfying $\sigma(K)=1_G$.
These maps are continuous.
We equip the space $G/K$ with the counting measure.
By the normalization, we have $|S|=|\pi^{-1}(S)|$ for every $S\subset G/K$.
For notational convenience, when evaluating a right $K$-invariant expression, we write $s\in G/K$ in place of its representative $\sigma(s)\in G$.
We also denote the subset $\{g\in G:gK\in S\}=\pi^{-1}(S)$ of $G$ by $\bigcup S$.

Recall that a locally compact group $G$ is \emph{amenable} if it admits a (left) invariant mean on $L^{\infty}(G)$.
When $G$ has a compact open subgroup $K<G$, amenability of $G$ is equivalent to the following F\o lner-type condition:
\begin{itemize}
    \item For every compact right $K$-invariant subset $F=FK\subset G$ and $\varepsilon>0$, there exists a nonempty finite subset $S\subset G/K$ that is $K$- and $(F,\varepsilon)$-invariant, in the sense that
    \[KS=S,\quad |S\Delta gS|<\varepsilon|S|\]
    for every $g\in F$.
\end{itemize}
We will always choose $F$ and $S$ so that $1_G\in F$ and $K\in S$, and we set $K_S=\bigcap_{s\in S}sKs^{-1}<K$.
Then $K_S$ is a compact open normal subgroup of $K$.

\subsection{Group actions}

An \emph{action} of a locally compact group $G$ on a von Neumann algebra $M$ is a continuous group homomorphism $\alpha\colon G\to \Aut(M)$.
More generally, a pair of Borel maps $\alpha\colon G\to \Aut(M)$ and $u\colon G\times G\to\U(M)$ is called a \emph{cocycle action} if
\begin{enumerate}
    \item $\alpha_g\circ\alpha_h=\Ad{u_{g,h}}\circ\alpha_{gh}$, and
    \item $u_{g,h}u_{gh,k}=\alpha_g(u_{h,k})u_{g,hk}$
\end{enumerate}
for all $g,h,k\in G$.
We also write $(\alpha,u)\colon G\curvearrowright M$ for such a cocycle action.
Cocycle actions are always assumed to be normalized, meaning that $u_{g,h}=1$ whenever $g=1_G$ or $h=1_G$.

These relations guarantee the existence of the \emph{(cocycle) crossed product} $M\rtimes_{\alpha,u}G$ (or, simply $M\rtimes G$) which is the von Neumann algebra generated by (a copy of) $M$ and unitaries $\lambda^{\alpha,u}_g$, $g\in G$ such that
\[\lambda_g^{\alpha,u}x=\alpha_g(x)\lambda_g^{\alpha,u},\quad\lambda_g^{\alpha,u}\lambda_h^{\alpha,u}=u_{g,h}\lambda_{gh}^{\alpha,u}\]
for $x\in M$ and $g,h\in G$.
If $(\pi,V)$ is a faithful covariant representation of $(\alpha,u)$, then the von Neumann algebra $(\pi(M)\wtensor\mathbb{C}1)\vee (V_g\otimes\lambda_g)_{g\in G}''$ is naturally isomorphic to $M\rtimes_{\alpha,u}G$, where $\lambda$ denotes the left regular representation of $G$.
See, for instance, \cite{Su80, Ta03a} for details.

For actions $\alpha\colon G\curvearrowright M$ and $\beta\colon G\curvearrowright N$, the diagonal action $\alpha\otimes\beta\colon G\curvearrowright M\wtensor N$ is defined by $(\alpha\otimes\beta)_g=\alpha_g\otimes\beta_g\in\Aut(M\wtensor N)$.
For a subset $\Theta$ of $\Aut(M)$, its fixed-point subalgebra is denoted by $M^{\Theta}=\{x\in M:\theta(x)=x\text{ for all }\theta\in\Theta\}$.
If an action $\alpha\colon G\curvearrowright M$ is fixed, we simply write $M^H$ instead of $M^{\alpha(H)}$ for a subgroup $H<G$.

Two cocycle actions $(\alpha,u)$ and $(\alpha',u')$ on $M$ are said to be \emph{cocycle conjugate} (resp.\ \emph{strongly cocycle conjugate}) if there exist an automorphism $\theta\in\Aut(M)$ (resp.\ $\theta\in\clInt(M)$) and a Borel map $v\colon G\to\U(M)$ such that $\theta\circ\Ad{v}\circ\alpha\circ\theta^{-1}=\alpha'$ and $v_g\alpha_g(v_h)u_{g,h}v_{gh}^{\ast}=\theta^{-1}(u_{g,h}')$ for all $g,h\in G$.
If $(\alpha,u)$ and $(\alpha',u')$ are cocycle conjugate, there exists a canonical isomorphism between inclusions $(M\subset M\rtimes_{\alpha,u}G)\cong(M\subset M\rtimes_{\alpha',u'}G)$ determined by $x\mapsto\theta(x)$ for $x\in M$ and $\lambda^{\alpha,u}_g\mapsto\theta(v_g^{\ast})\lambda^{\alpha',u'}_g$ for $g\in G$.
In particular, the 2-cocycle $u$ is called a \emph{coboundary} if $(\alpha,u)$ is cocycle conjugate to a genuine action.
In other words, $u$ is a coboundary if and only if there exists a Borel map $v\colon G\to\U(M)$ such that $v_g\alpha_g(v_h)u_{g,h}v_{gh}^{\ast}=1$ for all $g,h\in G$.
This is also equivalent to saying that the map $g\mapsto v_g\lambda^{\alpha,u}_g$ becomes a unitary representation of the group $G$.
For an action $\alpha$, an \emph{$\alpha$-cocycle} is a Borel map $v\colon G\to\U(M)$ satisfying the cocycle identity $v_g\alpha_g(v_h)=v_{gh}$.
An $\alpha$-cocycle $v$ is called a \emph{coboundary} if there exists a unitary element $w\in\U(M)$ such that $v_g=w\alpha_g(w^{\ast})$ for all $g\in G$.

\begin{prop}[{\cite[Theorem 4.1.3]{Su80}}]\label{prop:Sutherland}
    Let $(\alpha,u)$ be a cocycle action on a properly infinite von Neumann algebra.
    Then $u$ is a coboundary.
\end{prop}

As an application, for every cocycle action $(\alpha,u)\colon G\curvearrowright M$, the amplification $(\alpha\otimes\id,u\otimes 1)\colon G\curvearrowright M\wtensor\mathbb{B}$ is cocycle conjugate to a genuine action.

If $G\curvearrowright M$ is an action, there exists a canonical embedding of the group von Neumann algebra $\mathrm{L}(G)=(\lambda_g)_{g\in G}''\subset\mathbb{B}(L^2(G))$ into $M\rtimes G$.
For every closed subgroup $H<G$, there is also a canonical inclusion $M\rtimes H\hookrightarrow M\rtimes G$.
Provided that $M$ is a factor and the action is strictly outer, this inclusion has expectation if and only if $H$ is open in $G$ (\cite{BB18}).
In this case, the faithful normal conditional expectation $\mathbb{E}_H\colon M\rtimes G\to M\rtimes H$ is given by
\[\mathbb{E}_H(x\lambda_g)=\begin{cases}
    x\lambda_g&\text{if }g\in H,\\
    0&\text{if }g\notin H
\end{cases}\]
for $x\in M$ and $g\in G$.
For a compact open subgroup $K<G$, there exists a corresponding projection $p_K\in \mathrm{L}(G)\subset M\rtimes G$ defined by
\[p_K=\frac{1}{|K|}\int_K\lambda_k\,dk.\]
The following lemma is well-known to experts.
See, for instance, \cite[Lemma 4.2]{BB18}.

\begin{lem}
    For an action $K\curvearrowright M$ of a compact group $K$, the projection $p_K$ commutes with $M^K$, and $M^Kp_K=p_K(M\rtimes K)p_K$.
\end{lem}

For a faithful normal semifinite weight $\varphi$ on $M$, there exists a canonical faithful normal semifinite weight $\hat{\varphi}$ on $M\rtimes G$, which is called the \emph{dual weight} $\hat{\varphi}$ of $\varphi$. 
Its modular automorphism is given by
\[\sigma^{\hat{\varphi}}_t(x\lambda_g)=\Delta_G(g)^{it}\sigma^{\varphi}_t(x)[D\varphi\circ\alpha_g^{-1}:D\varphi]_t\lambda_g\]
for $x\in M$, $g\in G$ and $t\in\mathbb{R}$, where $\Delta_G$ is the modular function of $G$ and $[D\varphi\circ\alpha_g^{-1}:D\varphi]_t$ is the Connes cocycle (cf.\ \cite[Chapter X]{Ta03a}).

Suppose that $M$ has a tracial state $\tau$ which is invariant under the action, and $H<\ker{\Delta_G}$ is a closed subgroup (e.g., a compact subgroup).
Then $M\rtimes H$ is contained in the centralizer algebra $(M\rtimes G)_{\hat{\tau}}$ with respect to the dual weight $\hat{\tau}$.

There are several known conditions describing the outerness of group actions.
We first compare the conditions needed in this paper.

\begin{defn}[cf. \cite{Va01, BB18}]
    Let $\alpha\colon G\curvearrowright M$ be an action on a factor $M$.
    \begin{enumerate}
        \item The action $\alpha$ is called \emph{faithful} if the map $\alpha\colon G\to\Aut(M)$ is injective.
        \item The action $\alpha$ is called \emph{outer} if $\alpha_g$ is not inner for every $g\neq 1_G$.
        \item The action $\alpha$ is called \emph{minimal} if it is faithful and $(M^G)'\cap M=\mathbb{C}1$.
        \item The action $\alpha$ is called \emph{strictly outer} if $M'\cap (M\rtimes G)=\mathbb{C}1$.
        \item The action $\alpha$ is called \emph{properly outer relative to $K$} if every $g\in G$ for which there exists a nonzero $a\in M$ such that $ax=\alpha_g(x)a$ for all $x\in M^K$ belongs to $K$.
    \end{enumerate}
\end{defn}

Faithfulness, outerness, and strict outerness are defined in the same way for cocycle actions.
By definition, properties (1)--(4) pass to restrictions on any closed subgroup.
Note that $\alpha$ is faithful if and only if the family of functions $\{g\mapsto\varphi(\alpha_g(x)):\varphi\in M_{\ast},\ x\in M\}\subset L^{\infty}(G)$ generates $L^{\infty}(G)$ as a von Neumann algebra.

\begin{thm}[{\cite[Remark 3.17]{ILP98}}]
    Let $\gamma\colon K\curvearrowright M$ be a minimal action of a compact group $K$ on a factor $M$.
    If $\theta\in\Aut(M)$ fixes $M^K$ pointwise, then there exists $k\in K$ such that $\theta=\gamma_k$.
\end{thm}

\begin{lem}\label{lem:outerness}
    Let $\alpha\colon G\curvearrowright M$ be an action on a factor $M$.
    Consider the following conditions:
    \begin{enumerate}
        \item The action $\alpha$ is minimal;
        \item The action $\alpha$ is strictly outer;
        \item The action $\alpha$ is outer and $\alpha|_K$ is minimal;
        \item The action $\alpha$ is properly outer relative to $K$ and $\alpha|_K$ is minimal.
    \end{enumerate}
    Then \textup{(1)}$\implies$\textup{(2)}$\iff$\textup{(3)}$\iff$\textup{(4)}.
\end{lem}
\begin{proof}
    The implications (1)$\implies$(3) and (4)$\implies$(3) are immediate.
    The proof of (2)$\implies$(3) is the same as that of \cite[Proposition 4.4]{BB18}.
    We therefore suppose that (3) holds.
    Since $\alpha|_K$ is minimal, it is strictly outer by \cite[Proposition 6.2]{Va01}.
    
    If $g\in G$ and nonzero $a\in M$ satisfy $ax=\alpha_g(x)a$ for all $x\in M^K$, then $a^{\ast}a\in(M^K)'\cap M=\mathbb{C}1$ and $aa^{\ast}\in(M^{gKg^{-1}})'\cap M=\mathbb{C}1$.
    Thus, after scaling, we may assume that $a$ is unitary.
    Since $\Ad{a^{\ast}}\circ\alpha_g\in\Aut(M)$ fixes $M^K$ pointwise, there exists $k\in K$ such that $\Ad{a^{\ast}}\circ\alpha_g=\alpha_k$.
    It follows that $\alpha_{gk^{-1}}=\Ad{a}$.
    Since $\alpha$ is outer, we have $g=k\in K$.
    This proves (3)$\implies$(4).
    
    To show (2), let $x\in M'\cap (M\rtimes G)$.
    For every $g\in G$, the element $x_g=\mathbb{E}_K(\lambda_g^{\ast}x)\in M\rtimes K$ satisfies $x_gy=\alpha_g^{-1}(y)x_g$ for all $y\in M$.
    Then $x_g^{\ast}x_g,x_gx_g^{\ast}\in M'\cap(M\rtimes K)=\mathbb{C}1$.
    If $x_g\neq0$, we may assume $x_g$ is a unitary normalizing $M$.
    By \cite[Corollary 3.11]{BB18}, we have $x_g=b\lambda_k$ for some $b\in\U(M)$ and $k\in K$.
    It follows that $\alpha_g^{-1}=\Ad{x_g}=\Ad{b}\circ\alpha_k$ on $M$.
    Since $\alpha$ is outer, we obtain $g=k^{-1}\in K$.
    Thus, the support of $x$ is contained in $K$ in the sense of \cite[Section 3]{BB18}.
    By \cite[Theorem 3.7]{BB18}, we conclude that $x\in M'\cap(M\rtimes K)=\mathbb{C}1$.
\end{proof}

\begin{prop}[cf.\ {\cite[Proposition 10]{MV23}}]
    Let $\alpha\colon G\curvearrowright M$ and $\beta\colon G\curvearrowright N$ be actions on factors.
    If $\beta$ is strictly outer, then the diagonal action $\alpha\otimes\beta$ is strictly outer.
\end{prop}

\subsection{Ultraproduct von Neumann algebras}

Our standard references for ultraproduct algebras are \cite[Chapter 5]{Oc85} and \cite{AH14}.
Let $M$ be a von Neumann algebra with separable predual, and fix a free ultrafilter $\omega$ on $\mathbb{N}$.
Let us consider the following \C-algebras of $\ell^{\infty}(\mathbb{N},M)$:
\begin{align*}
    \mathscr{T}_{\omega}(M) &= \{(x_n)_n:x_n\to 0\text{ $*$-strongly as }n\to\omega\},\\
    \mathscr{C}_{\omega}(M) &= \{(x_n)_n:\|[x_n,\varphi]\|\to 0\text{ as }n\to\omega\text{ for all }\varphi\in M_{\ast}\},\\
    \mathscr{N}_{\omega}(M) &= \{(x_n)_n:(x_n)_n\text{ normalizes }\mathscr{T}_{\omega}(M)\text{ in }\ell^{\infty}(\mathbb{N},M)\}.
\end{align*}
Then $\mathscr{T}_{\omega}(M)\subset\mathscr{C}_{\omega}(M)\subset\mathscr{N}_{\omega}(M)$.
Note that if $M$ is tracial then $\mathscr{N}_{\omega}(M)=\ell^{\infty}(\mathbb{N},M)$.
We call the quotient \C-algebras $M^{\omega}=\mathscr{N}_{\omega}(M)/\mathscr{T}_{\omega}(M)$ (resp.\ $M_{\omega}=\mathscr{C}_{\omega}(M)/\mathscr{T}_{\omega}(M)$) the \emph{ultraproduct algebra} (resp.\ the \emph{central sequence algebra}).
They are in fact von Neumann algebras.
The image of a sequence $(x_n)_n\in\mathscr{N}_{\omega}(M)$ in $M^{\omega}$ is denoted by $(x_n)^{\omega}$.
We regard $M$ and $M_{\omega}$ as von Neumann subalgebras of $M^{\omega}$.
Under these identifications, we have the inclusion $M_{\omega}\subset M'\cap M^{\omega}$.
When $M$ is tracial, equality holds.
All arguments below are independent of the choice of ultrafilter $\omega$.

The formula $\mathbb{E}_{\omega}((x_n)^{\omega})=\wlim_{n\to\omega}x_n$ defines a faithful normal conditional expectation $\mathbb{E}_{\omega}$ from $M^{\omega}$ onto $M$.
Its restriction $\tau_{\omega}\coloneq \mathbb{E}_{\omega}|_{M_{\omega}}$ is a faithful normal trace with values in $\Z(M)$.
In particular, $\tau_{\omega}$ is a tracial state on $M_{\omega}$ if $M$ is a factor.
For $\varphi\in M_{\ast}$, we write $\varphi^{\omega}=\varphi\circ\mathbb{E}_{\omega}\in(M^{\omega})_{\ast}$.
Every automorphism $\theta\in\Aut(M)$ induces $\theta^{\omega}\in\Aut(M^{\omega})$ given by $\theta^{\omega}((x_n)^{\omega})=(\theta(x_n))^{\omega}$.
This restricts to $\theta^{\omega}\in\Aut(M_{\omega})$.
We often write $\theta$ instead of $\theta^{\omega}$ when no confusion arises.

\begin{defn}[equicontinuous parts, cf.\ {\cite[Definitions 3.1, 3.4 and 3.9]{MT16}}]
    Let $X$ be a Polish space equipped with a Radon measure $\mu$, $M$ a von Neumann algebra with a fixed faithful normal state $\varphi$ and $\omega$ a free ultrafilter on $\mathbb{N}$.
    \begin{enumerate}
        \item A sequence of maps $(f_n\colon X\to M)_n$ is said to be \emph{$\omega$-equicontinuous} if for every $x\in X$ and $\delta>0$, there exists a neighborhood $U$ of $x\in X$ such that
        \[\{n\in\mathbb{N}:\|f_n(x)-f_n(y)\|_{\varphi}^{\sharp}<\delta\text{ whenever }y\in U\}\in\omega.\]
        \item Let $\alpha\colon X\to\Aut(M)$ be a Borel map.
        A sequence $(a_n)_n\in\mathscr{N}_{\omega}(M)$ is said to be \emph{$(\alpha,\omega)$-equicontinuous} if for every Borel set $E\subset X$ with $\mu(E)<\infty$ and $\delta>0$, there exists a compact subset $K\subset E$ such that
        \begin{itemize}
            \item $\alpha|_K$ is continuous;
            \item $\mu(E\setminus K)<\delta$;
            \item the sequence $(K\ni x\mapsto \alpha_x(a_n)\in M)_n$ is $\omega$-equicontinuous.
        \end{itemize}
    \end{enumerate}
    Let $\mathscr{E}_{\alpha,\omega}(M)$ denote the set of all $(\alpha,\omega)$-equicontinuous sequences in $\mathscr{N}_{\omega}(M)$, and then $\mathscr{E}_{\alpha,\omega}(M)$ is a \C-subalgebra containing $\mathscr{T}_{\omega}(M)$.
    We set
    \[M^{\omega}_{\alpha}\coloneq \mathscr{E}_{\alpha,\omega}(M)/\mathscr{T}_{\omega}(M),\quad M_{\omega,\alpha}\coloneq(\mathscr{E}_{\alpha,\omega}(M)\cap\mathscr{C}_{\omega}(M))/\mathscr{T}_{\omega}(M)=M^{\omega}_{\alpha}\cap M_{\omega}\]
    and we say that $M^{\omega}_{\alpha}$ and $M_{\omega,\alpha}$ are the \emph{$(\alpha,\omega)$-equicontinuous parts} of $M^{\omega}$ and $M_{\omega}$, respectively.
\end{defn}

This definition does not depend on the choice of $\varphi$, and $\|\cdot\|_{\varphi}^{\sharp}$ may be replaced by the 1-norm $|\cdot|_{\varphi}^{\sharp}$.
The \C-subalgebras $M^{\omega}_{\alpha}$ and $M_{\omega,\alpha}$ are von Neumann algebras.
Lusin's theorem implies that $M\subset M^{\omega}_{\alpha}$.
Note that if $\alpha\colon G\to\Aut(M)$ is an action, then $(a_n)_n\in\mathscr{N}_{\omega}(M)$ is $(\alpha,\omega)$-equicontinuous if and only if for every $\delta>0$, there exists a neighborhood $U$ of $1_G\in G$ such that
\[\{n\in\mathbb{N}:\|a_n-\alpha_g(a_n)\|_{\varphi}^{\sharp}<\delta\text{ whenever }g\in U\}\in\omega.\]
In particular, we have $M^{\omega}_{\alpha}=M^{\omega}_{\alpha|_U}$ for every open subgroup $U<G$.
A cocycle action $(\alpha,u)\colon G\curvearrowright M$ induces an action on $M_{\omega,\alpha}$.
See also \cite{MT16, To17}.

\begin{lem}
    Let $\gamma\colon K\curvearrowright M$ be an action of a compact group $K$.
    Then $(M^{\omega}_{\gamma})^K=(M^K)^{\omega}$ and $(M_{\omega,\gamma})^K=M_{\omega}\cap(M^K)^{\omega}=M'\cap(M^K)_{\omega}$ inside $M^{\omega}$.
\end{lem}
\begin{proof}
    For $x=(x_n)^{\omega}\in(M^{\omega}_{\gamma})^K$, equicontinuity implies that
    \[x=\int_K\gamma_k((x_n)^{\omega})\,dk=\int_K(\gamma_k(x_n))^{\omega}\,dk=\left(\int_K\gamma_k(x_n)\,dk\right)^{\omega}\in (M^K)^{\omega}.\]
    Here, the integral is taken with respect to the normalized Haar measure.
    The reverse inclusion is immediate, hence $(M^{\omega}_{\gamma})^K=(M^K)^{\omega}$.
    This also implies that $(M_{\omega,\gamma})^K=M_{\omega}\cap(M^{\omega}_{\gamma})^K=M_{\omega}\cap(M^K)^{\omega}\subset M'\cap(M^K)_{\omega}$.
    Since there exists a faithful normal conditional expectation from $M$ onto $M^K$, we have $M'\cap(M^K)_{\omega}\subset M_{\omega}$.
    This completes the proof.
\end{proof}

The following elementary compactness argument is an ultrafilter version of the Arzel\`a--Ascoli theorem.
It allows us to interchange ultralimits with integration over compact spaces.

\begin{lem}[cf.\ {\cite[Lemma 3.3]{MT16}}]
    Let $X$ be a compact space and $f_n\colon X\to\mathbb{C}$ be a uniformly bounded $\omega$-equicontinuous sequence of Borel maps.
    Then the convergence $\lim_{n\to\omega}f_n$ is uniform.
\end{lem}

Suppose now that $M$ is a factor and $\varphi$ is a faithful normal state on $M$, and let $L^{\infty}(X)\subset M_{\omega}$ be an abelian von Neumann subalgebra, where $X$ is a compact metric space.
We consider the probability measure $\mu=\tau_{\omega}|_{L^{\infty}(X)}$ on $X$.

\begin{lem}
    In the above setting, for $a\in L^{\infty}(X)'\cap M^{\omega}$ and $f\in L^{\infty}(X)$, one has $|af|_{\varphi^{\omega}}^{\sharp}\le\|a\||f|_1$.
    If $N\subset M^{\omega}$ is globally $\sigma^{\varphi^{\omega}}$-invariant and $\varphi^{\omega}$-orthogonal to $L^{\infty}(X)$, then $|F|_{\varphi^{\omega}}^{\sharp}\le\esssup_{x\in X}|F(x)|_{\varphi^{\omega}}^{\sharp}$ for $F\in L^{\infty}(X)\wtensor N\subset M^{\omega}$.
\end{lem}
\begin{proof}
    Since the modular flow $\sigma^{\varphi^{\omega}}$ leaves $L^{\infty}(X)$ globally invariant, there exists a $\varphi^{\omega}$-preserving normal conditional expectation $\mathbb{E}\colon M^{\omega}\to L^{\infty}(X)$.
    Lemma \ref{lem:1-norm}(1) gives that $|af|_{\varphi^{\omega}}=\|af\varphi^{\omega}\|\le\|a\||f|_{\varphi^{\omega}}=\|a\||f|_1$.
    Since $a$ and $f$ commute, it follows that $|af|_{\varphi^{\omega}}^{\sharp}\le\|a\||f|_1$.

    To show the second assertion, we may assume that $N$ is globally $\sigma^{\varphi^{\omega}}$-invariant.
    Lemma \ref{lem:1-norm}(2) implies
    \[|F|_{\varphi^{\omega}}^{\sharp}=\int_X|F(x)|_{\varphi^{\omega}}^{\sharp}d\mu(x)\le\esssup_{x\in X}|F(x)|_{\varphi^{\omega}}^{\sharp}\]
    for every $F\in L^{\infty}(X)\wtensor N$.
\end{proof}

The following two techniques due to Ocneanu are standard tools in the analysis of ultraproducts.
We will not use the slow reindexation trick (\cite[Lemma 5.4]{Oc85}).

\begin{lem}[Fast reindexation trick, {cf.\ \cite[Lemma 5.3]{Oc85}}]
    Let $M$ be a factor, let $\alpha\colon G\curvearrowright M$ be an action, and let $N_0\subset M^{\omega}_{\alpha}$ and $N_1\subset M_{\omega,\alpha}$ be countably generated von Neumann subalgebras.
    Suppose that $N_0$ and $N_1$ are globally $\alpha$-invariant.
    Then there exists a normal embedding $\Theta\colon N_0\wtensor N_1\to M^{\omega}_{\alpha}$ such that
    \begin{enumerate}
        \item $\Theta(x\otimes 1)=x$ for all $x\in N_0$;
        \item $\varphi^{\omega}\circ\Theta=\varphi^{\omega}\otimes\tau_{\omega}$ on $N_0\otimes N_1$ for all $\varphi\in M_{\ast}$;
        \item $\alpha_g\circ\Theta=\Theta\circ(\alpha\otimes\alpha)_g$ on $N_0\otimes N_1$ for all $g\in G$.
    \end{enumerate}
\end{lem}

\begin{lem}[Index selection trick, {cf.\ \cite[Lemma 5.5]{Oc85}}]
    Let $M$ be a von Neumann algebra, $N$ a countably generated von Neumann subalgebra of $M^{\omega}$, and $P$ a countably generated subalgebra of $N^{\omega}$.
    Suppose in addition that either $N\subset M_{\omega}$ or $M$ is finite.
    Then there exists a normal embedding $\Theta\colon P\to M^{\omega}$ such that
    \begin{enumerate}
        \item $\Theta(x)=x$ for all $x\in N\cap P$;
        \item $\varphi^{\omega}\circ\Theta=(\varphi^{\omega}|_N)^{\omega}$ on $P$ for all $\varphi\in M_{\ast}$;
        \item $\Theta((x_n)^{\omega})\in M_{\omega}$ for $(x_n)^{\omega}\in P$ if $x_n\in M_{\omega}$ for all $n$;
        \item $\theta^{\omega}\circ\Theta=\Theta\circ(\theta^{\omega}|_N)^{\omega}$ on $P$ if $\theta\in\Aut(M)$ leaves $N$ and $P$ globally invariant.
    \end{enumerate}
    Suppose moreover that $(\alpha,u)\colon G\curvearrowright M$ is a cocycle action and that $N\subset M^{\omega}_{\alpha}$ is globally invariant under $\alpha$.
    If $P\subset N^{\omega}_{\alpha}$, then $\Theta$ provides the embedding $P\hookrightarrow M^{\omega}_{\alpha}$.
\end{lem}

The additional assumption that either $N\subset M_{\omega}$ or $M$ is finite ensures that the index selection is carried out so that the selected sequences belong to $\N_{\omega}(M)$.

\subsection{Inclusions of von Neumann algebras}

All inclusions considered in this paper are unital.
For an inclusion of von Neumann algebras $(N\subset M)$, the \emph{normalizer} $\N_M(N)$ is the closed subgroup of the unitary group $\U(M)$ defined by
\[\N_M(N)=\{u\in\U(M):uNu^{\ast}=N\}.\]
Then $\U(N)$ is a normal subgroup of $\N_M(N)$.

\begin{defn}
    Let $(N\subset M)$ be an inclusion of von Neumann algebras.
    \begin{enumerate}
        \item The inclusion $(N\subset M)$ is called \emph{regular} if $\N_M(N)''=M$.
        \item The inclusion $(N\subset M)$ is called \emph{irreducible} if $N'\cap M=\mathbb{C}1$.
        \item The inclusion $(N\subset M)$ is said to \emph{have expectation} if there is a faithful normal conditional expectation from $M$ onto $N$.
    \end{enumerate}
\end{defn}

\begin{exam}
    Let $G\curvearrowright M$ be a cocycle action.
    Then the inclusion $(M\subset M\rtimes G)$ is always regular, and it is irreducible if and only if $M$ is a factor and $G\curvearrowright M$ is strictly outer.
\end{exam}

The \emph{relative automorphism group} of an inclusion $(N\subset M)$ is the closed subgroup of $\Aut(M)$ defined by
\[\Aut(N\subset M)=\{\theta\in\Aut(M):\theta(N)=N\}\]
and its \emph{relative inner automorphism group} is
\[\Int(N\subset M)=\{\Ad{u}\in\Aut(M):u\in\U(N)\}.\]
Note that $\Int(N\subset M)\neq\Int(M)\cap\Aut(N\subset M)=\{\Ad{u}:u\in\N_M(N)\}$ in general.
For instance, if $G\curvearrowright M$ is outer and $1_G\neq g\in G$, then $\Ad{\lambda_g}\in\Int(M\rtimes G)\cap\Aut(M\subset M\rtimes G)$ while $\Ad{\lambda_g}\notin\Int(M\subset M\rtimes G)$.

Two inclusions $(N_1\subset M_1)$ and $(N_2\subset M_2)$ are said to be \emph{isomorphic} if there exists an isomorphism $\theta\colon M_1\cong M_2$ such that $\theta(N_1)=N_2$.
Recall that cocycle conjugacy of two cocycle actions $(\alpha,u)$ and $(\alpha',u')$ always induces an isomorphism between their crossed product inclusions $(M\subset M\rtimes_{\alpha,u}G)\cong(M\subset M\rtimes_{\alpha',u'}G)$.

\begin{defn}[cf.\ \cite{BB18}]
    A strictly outer action $G\curvearrowright M$ of a locally compact group on a factor $M$ is said to satisfy the \emph{Intermediate Subfactor Property} (in short, \emph{ISP}) if any subfactor $P$ of $M\rtimes G$ containing $M$ is of the form $P=M\rtimes H$ for some closed subgroup $H<G$.
\end{defn}

When the ISP holds, the subgroup $H$ is unique and is recovered from $P$ by $H=\{g\in G:\lambda_g\in P\}$.
Thus $H\mapsto M\rtimes H$ gives a bijection from the closed subgroups of $G$ onto the intermediate subfactors of $M\subset M\rtimes G$.

\begin{thm}\label{thm:ISP}
    Every strictly outer action of a totally disconnected locally compact group on a factor satisfies the ISP.
\end{thm}
\begin{proof}
    Let $\alpha$ be a strictly outer action of a totally disconnected group $G$ on a factor.
    By Lemma \ref{lem:outerness}, the action $\alpha$ is properly outer relative to a compact open subgroup $K<G$ and $\alpha|_K$ is minimal.
    When $M$ is semifinite, this was proven in \cite[Theorem A]{BB18}, and the general case is due to \cite[Theorem F]{Is25}.
\end{proof}

\section{Centrally free actions}\label{sec:2}

\begin{defn}
    A cocycle action $(\alpha,u)\colon G\curvearrowright M$ of a locally compact group $G$ on a factor $M$ is called \emph{centrally free} if the induced action $G\curvearrowright M_{\omega,\alpha}$ is faithful.
\end{defn}

The fast reindexation trick shows that central freeness is equivalent to the condition that the induced action $G\curvearrowright (N_0'\cap M_{\omega,\alpha})$ is properly outer for every globally $G$-invariant countably generated subalgebra $N_0\subset M^{\omega}_{\alpha}$.

\subsection{Outerness conditions}

We will compare several outerness conditions with central freeness, and examine their behavior on central sequence algebras.

\begin{lem}[cf. {\cite[Proposition 3.7]{Ni26}}]\label{lem:cpt-Rohlin}
    Let $\gamma\colon K\curvearrowright M$ be an action of a compact group $K$ on a factor $M$.
    Consider the following conditions:
    \begin{enumerate}
        \item The action $\gamma$ has the Rohlin property, that is, there exists an equivariant embedding $L^{\infty}(K)\hookrightarrow M_{\omega,\gamma}$;
        \item The action $\gamma$ is minimal;
        \item The action $\gamma$ is strictly outer.
    \end{enumerate}
    Then \textup{(1)}$\implies$\textup{(2)}$\iff$\textup{(3)}.
    If $M=R$, then all these conditions are equivalent.
\end{lem}

To derive the Rohlin property from central freeness, we realize a faithful Gaussian $K$-action inside the central sequence algebra.
Recall that the Gaussian construction associated with a real Hilbert space $\mathcal{H}_{\mathbb{R}}$ consists of an abelian von Neumann algebra $(L^{\infty}(\hat{\mathcal{H}}),\tau)$ and a map $W\colon\mathcal{H}_{\mathbb{R}}\to\U(L^{\infty}(\hat{\mathcal{H}}))$ such that $W(\xi+\eta)=W(\xi)W(\eta)$ and $\tau(W(\xi))=\exp(-\frac12\|\xi\|^2)$.
Writing $W(\xi)=\exp(i\hat{\xi})$, the family $(\hat{\xi})_{\xi\in\mathcal{H}_{\mathbb{R}}}$ is the associated Gaussian process.
See \cite[Section II.1]{Bo14}.

\begin{lem}
    Let $x,y$ be self-adjoint elements in a tracial \C-algebra.
    \begin{enumerate}
        \item $\|e^{ix}e^{iy}-e^{i(x+y)}\|_2\le\frac12\|[x,y]\|_2$.
        \item $\|e^{ix}-1\|_2\le\|x\|_2$.
    \end{enumerate}
\end{lem}
\begin{proof}
    Assertion (1) is the case of two variables in \cite[Lemma 4]{Ku05}, applied to the 2-norm.
    For (2), use
    \[e^{ix}-1=\int_0^1ixe^{itx}\,dt\]
    and take the 2-norm.
\end{proof}

\begin{lem}\label{lem:Gaussian}
    Let $M$ be a factor and let $\mathcal{H}_{\mathbb{R}}\subset (M_{\omega})_{\sa}\ominus\mathbb{R}1\subset L^2(M_{\omega})_{\mathbb{R}}$ be a separable real pre-Hilbert subspace.
    Then there exists a map $W\colon\mathcal{H}_{\mathbb{R}}\to\U(M_{\omega})$ such that
    \begin{enumerate}
        \item $W(\xi+\eta)=W(\xi)W(\eta)$ for all $\xi,\eta\in\mathcal{H}_{\mathbb{R}}$;
        \item $\tau_{\omega}(W(\xi))=\exp(-\frac12\|\xi\|_2^2)$;
        \item If $\theta\in\Aut(M)$ leaves $\mathcal{H}_{\mathbb{R}}$ globally invariant (in other words, $\theta$ induces an orthogonal transformation on $\mathcal{H}_{\mathbb{R}}$), then $\theta(W(\xi))=W(\theta(\xi))$ for all $\xi\in\mathcal{H}_{\mathbb{R}}$.
    \end{enumerate}
\end{lem}
\begin{proof}
    Let $N\subset M_{\omega}$ be a countably generated von Neumann subalgebra containing $\mathcal{H}_{\mathbb{R}}$.
    By reindexation, we may realize $\bigwtensor_{j=1}^{\infty}N$ as a von Neumann subalgebra of $M_{\omega}$.
    Hence there are countably many mutually orthogonal copies of $\mathcal{H}_{\mathbb{R}}$, denoted by $\mathcal{H}_{\mathbb{R}}^{(j)}$, $j\in\mathbb{N}$.
    For $\xi\in\mathcal{H}_{\mathbb{R}}$, we write $\xi^{(j)}\in\mathcal{H}_{\mathbb{R}}^{(j)}$ for its copy.
    Set $S_n(\xi)=\frac1{\sqrt{n}}\sum_{j=1}^n\xi^{(j)}$ for $\xi\in\mathcal{H}_{\mathbb{R}}$.
    Applying the index selection trick, we represent $W(\xi)=(\exp(iS_n(\xi)))^{\omega}\in (\bigwtensor N)^{\omega}$ inside of $M_{\omega}$.
    We will show that this map $\xi\mapsto W(\xi)$ satisfies the desired properties.
    
    Since $\xi^{(j)}$ and $\eta^{(k)}$ commute if $j\neq k$, it follows that $[S_n(\xi),S_n(\eta)]=\frac1n\sum_{j=1}^n[\xi^{(j)},\eta^{(j)}]$ and $\|[S_n(\xi),S_n(\eta)]\|_2^2=\frac1{n^2}\sum_{j=1}^n\|[\xi,\eta]\|_2^2=\frac1n\|[\xi,\eta]\|_2^2$.
    The preceding lemma gives
    \[\|e^{iS_n(\xi)}e^{iS_n(\eta)}-e^{iS_n(\xi+\eta)}\|_2\le\frac1{2\sqrt{n}}\|[\xi,\eta]\|_2.\]
    Thus $W$ satisfies item (1).
    Since $\tau_{\omega}(\xi)=0$, we have $\tau_{\omega}(\exp(iS_n(\xi)))=\tau_{\omega}(\exp(\frac1{\sqrt{n}}i\xi))^n\to\exp(-\frac12\|\xi\|_2^2)$.
    This proves (2), while item (3) is clear by construction.
\end{proof}

\begin{prop}\label{prop:cpt-cent.free->Rohlin}
    A centrally free (cocycle) action $\gamma\colon K\curvearrowright M$ of a compact group $K$ on a factor $M$ has the Rohlin property.
    % Moreover, $(N^K)'\cap N=\Z(N)$.
\end{prop}
\begin{proof}
    We write $N=M_{\omega,\gamma}$.
    Since $\gamma$ is centrally free, we obtain a separable real pre-Hilbert subspace $\mathcal{H}_{0,\mathbb{R}}\subset N_{\sa}\ominus\mathbb{R}1\subset L^2(N)_{\mathbb{R}}$ on which $K$ acts faithfully.
    By reindexation as above, we realize countably many mutually orthogonal copies of this subspace inside $N$, denoted by $\mathcal{H}_{j,\mathbb{R}}$, $j\in\mathbb{N}$.
    Their algebraic direct sum $\mathcal{H}_{\mathbb{R}}\subset N_{\sa}\ominus\mathbb{R}1$ is a separable real pre-Hilbert subspace such that the induced orthogonal representation $K\curvearrowright\mathcal{H}_{\mathbb{R}}$ is faithful and has infinite multiplicity.
    Lemma \ref{lem:Gaussian} provides a $K$-equivariant Gaussian map $W\colon\mathcal{H}_{\mathbb{R}}\to\U(M_{\omega})$.
    Then $W(\xi)$ belongs to the equicontinuous part $M_{\omega,\gamma}$ because
    \[\|\gamma_k(\exp(iS_n(\xi)))-\exp(iS_n(\xi))\|_2\le\|S_n(\gamma_k(\xi)-\xi)\|_2=\|\gamma_k(\xi)-\xi\|_2.\]
    Thus, there exists a $K$-equivariant embedding $L^{\infty}(\hat{\mathcal{H}})\hookrightarrow N=M_{\omega,\gamma}$.
    
    Since $K\curvearrowright\mathcal{H}_{\mathbb{R}}$ is faithful and has infinite multiplicity, the Gaussian action $K\curvearrowright\hat{\mathcal{H}}$ is essentially free, and thus disintegration of the invariant measure along the orbits identifies $\hat{\mathcal{H}}$ with $K\times K\backslash\hat{\mathcal{H}}$ as a $K$-space modulo null sets.
    Consequently, we obtain a $K$-equivariant embedding $L^{\infty}(K)\hookrightarrow L^{\infty}(\hat{\mathcal{H}})\hookrightarrow N$.
\end{proof}

In particular, every centrally free action $\alpha\colon G\curvearrowright M$ is strictly outer.
Indeed, central freeness implies that $\alpha$ is outer and that $\alpha|_K$ is minimal by Lemma \ref{lem:cpt-Rohlin}, so Lemma \ref{lem:outerness} yields the conclusion.

\begin{lem}
    If $\alpha\colon G\curvearrowright M$ is centrally free, then $G\curvearrowright N$ is strictly outer $N'\cap (N\rtimes G)\subset N$, where $N=M_{\omega,\alpha}$.
    In particular, $(p_K(N\rtimes K)p_K)'\cap (p_K(N\rtimes G)p_K)\subset p_K(N\rtimes K)p_K=N^Kp_K$.
\end{lem}
\begin{proof}
    In the proof, we identify the crossed product $N\rtimes G$ with the von Neumann subalgebra of $N\wtensor\mathbb{B}(L^2(G))$ generated by $N\cong\{g\mapsto\alpha_{g^{-1}}(x):x\in N\}\subset N\wtensor L^{\infty}(G)$ and $\mathbb{C}1\wtensor\mathrm{L}(G)$.

    Let $x\in N'\cap(N\rtimes G)$.
    There exist countably generated, mutually orthogonal, globally $G$-invariant subalgebras $N_0,N_1\subset N$ such that $x\in N_0\rtimes G$ and $G\curvearrowright N_1$ is faithful.
    Inside $(N_0\wtensor N_1)\rtimes G$, we have
    \[x\in (\mathbb{C}1\wtensor N_1)'\cap((N_0\wtensor\mathbb{C}1)\rtimes G)\subset (N_0\wtensor N_1)\rtimes G.\]
    Since $G\curvearrowright N_1$ is faithful, the slices $\bigcup\{\varphi\otimes\id(N_1):\varphi\in(N_1)_{\ast}\}\subset L^{\infty}(G)$ generate $L^{\infty}(G)$.
    Thus
    \[(\mathbb{C}1\wtensor N_1)'\cap((N_0\wtensor\mathbb{C}1)\rtimes G)\subset(N_0\wtensor L^{\infty}(G))\cap (N_0\rtimes G)=N_0.\]
    The second assertion follows easily from this.
\end{proof}

% \begin{prop}\label{prop:cent.free->prop.outer}
%     If $\alpha\colon G\curvearrowright M$ is centrally free, then $\alpha|_K$ is minimal and $\alpha$ is properly outer relative to $K$.
%     In particular, $\alpha$ is strictly outer in this case.
% \end{prop}
% \begin{proof}
%     Since $\alpha|_K$ is centrally free, it has the Rohlin property and hence $\alpha|_K$ is minimal.
%     Suppose $g\in G$ and nonzero $a\in M$ satisfy that $ax=\alpha_g(x)a$ for all $x\in M^K$.
%     This relation $ay=\alpha_g(y)a$ also holds for every $y\in(M_{\omega,\alpha})^K$, thus, $\alpha_g(y)=y$ for every $y\in(M_{\omega,\alpha})^K$.
%     We have $\lambda_g\in (N^K)'\cap (N\rtimes G)$, where $N=M_{\omega,\alpha}$.
%     It follows that $p_K\lambda_gp_K\in N^Kp_K$ and hence $\mathbb{E}_K(p_K\lambda_gp_K)=p_K\lambda_gp_K$, which means $g\in K$.
%     % $p_K\lambda_gp_K\neq0$ for all $g\in G$.
% \end{proof}

\subsection{Actions on \texorpdfstring{$R$}{R}}

In this subsection, we show that every strictly outer action of a (not necessarily amenable) locally compact group $G$ with a compact open subgroup $K<G$ on the hyperfinite $\II_1$ factor $R$ is centrally free.
For discrete group actions, central freeness is equivalent to pointwise-outerness, since every centrally trivial automorphism of $R$ is inner (see \cite{Co77}).
By contrast, for continuous groups, it is more difficult to see the central freeness.

\begin{lem}
    Let $\gamma\colon K\curvearrowright R$ be \emph{the} minimal action of a compact group.
    Then the induced action $\gamma\colon K\curvearrowright R_{\omega,\gamma}$ is minimal.
\end{lem}
\begin{proof}
    We will show the stronger statement $(R'\cap (R^K)^{\omega})'\cap R_{\omega}=\mathbb{C}1$.
    By the uniqueness of minimal actions (\cite{MT07}), we may use the infinite tensor product model $R=\bigwtensor_j\mathbb{M}_{d_j}$ and $\gamma=\bigotimes_j\Ad{v_j}$, where $\Ad{v_j}$ is a projective representation on $\mathbb{M}_{d_j}$.
    Let $R_n=\bigwtensor_{j\le n}\mathbb{M}_{d_j}$.
    Note that $\gamma|_{R_n'\cap R}$ is also minimal.
    Let $x=(x_n)^{\omega}\in R_{\omega}$ be a nontrivial central sequence.
    We may assume that $x_n\in R_n'\cap R$, $\tau(x_n)=0$, and $\|x_n\|_2=1$.

    We claim that $\sup\{\|x_n-ux_nu^{\ast}\|_2:u\in\U((R_n'\cap R)^K)\}\ge 1$.
    Indeed, the unique element in $\clco\{ux_nu^{\ast}:u\in\U((R_n'\cap R)^K)\}$ of minimal 2-norm is invariant under conjugation by unitaries in $(R_n'\cap R)^K$, and therefore it belongs to $((R_n'\cap R)^K)'\cap(R_n'\cap R)=\mathbb{C}1$.
    Since $x_n$ has trace zero, it must be zero and we obtain the claim.
    
    We find a unitary $u_n\in (R_n'\cap R)^K=R_n'\cap R^K$ such that $\|x_n-u_nx_nu_n^{\ast}\|_2\ge 1/2$.
    Then the element $u=(u_n)^{\omega}\in R'\cap (R^K)^{\omega}$ does not commute with $x$.
    This completes the proof.
\end{proof}

\begin{rem}\label{rem:need-minimality}
    As a consequence, $R_{\omega,\gamma}$ is a factor.
    This conclusion may fail without minimality.
    Indeed, consider a free ergodic p.m.p.\ action $\mathbb{Z}\curvearrowright X$ on a standard probability space.
    Then the crossed product $L^{\infty}(X)\rtimes\mathbb{Z}$ is isomorphic to $R$, and the dual action $\gamma\colon\mathbb{T}\curvearrowright R$ is not minimal.
    By \cite[Theorem 3.6(2)]{To17}, there exists a canonical isomorphism $R^{\omega}_{\gamma}=L^{\infty}(X)^{\omega}\rtimes\mathbb{Z}$, and $R_{\omega,\gamma}=R'\cap R^{\omega}_{\gamma}=(L^{\infty}(X)^{\omega})^{\mathbb{Z}}$ is not a factor.
    Moreover, although $R$ is McDuff, $R_{\omega,\gamma}$ does not contain a copy of $R$.

    If, for instance, the original action $\mathbb{Z}\curvearrowright X$ is weakly mixing (e.g., a Bernoulli shift), then the dual action $\gamma$ is outer.
\end{rem}

\begin{prop}\label{prop:str.outer->cent.free}
    Let $\alpha\colon G\curvearrowright R$ be a strictly outer action of a locally compact group with a compact open subgroup $K<G$.
    Then $\alpha$ is centrally free.
\end{prop}
\begin{proof}
    We already know that $\alpha|_K$ is minimal and has the Rohlin property, hence $K\curvearrowright R_{\omega,\alpha}$ is faithful.
    Let $g\in G\setminus K$ and suppose for contradiction that $\alpha_g$ acts trivially on $R_{\omega,\alpha}$.
    Since $\alpha_k=\alpha_{gkg^{-1}}$ on $R_{\omega,\alpha}$ for every $k\in K\cap g^{-1}Kg$, we have $k=gkg^{-1}$ and $g$ centralizes $K\cap g^{-1}Kg$ in $G$.
    Replacing $K$ by $K\cap g^{-1}Kg$, we assume that $g$ commutes with $K$.
    
    Since $\alpha_g\in\Aut(R)$ is outer, there exist $x=(x_n)^{\omega}\in R_{\omega}$ and $\delta>0$ such that $\|x\|_2=1$ and $\|x-\alpha_g(x)\|_2>\delta$.
    We may assume that $\|x_n\|_2=1$ and $\|x_n-\alpha_g(x_n)\|_2>\delta$ for every $n$.
    By the Rohlin property of $\alpha|_K$, define $f_n\in L^{\infty}(K)\wtensor R\subset R^{\omega}_{\alpha}$ by $f_n(k)=\alpha_k(x_n)$ for each $n$.
    For $k,l\in K$, we have
    \[[\alpha_k(f_n)](l)=\alpha_k(f_n(k^{-1}l))=\alpha_l(x_n)=f_n(l)\]
    and hence $f_n\in (R^{\omega}_{\alpha})^K=(R^K)^{\omega}$.
    Since $\alpha_g$ fixes $L^{\infty}(K)\subset R_{\omega,\alpha}$ pointwise, we have $\alpha_g(f_n)\in L^{\infty}(K)\wtensor R$ and $[\alpha_g(f_n)](k)=\alpha_g(f_n(k))=\alpha_k(\alpha_g(x_n))$.
    Therefore, it follows that $\|f_n-\alpha_g(f_n)\|_2=\|x_n-\alpha_g(x_n)\|_2>\delta$.
    The index selection trick gives an element $y=(f_n)^{\omega}\in R'\cap(R^K)^{\omega}\subset R_{\omega,\alpha}$ so that $\|y\|_2=1$ and $\|y-\alpha_g(y)\|_2\ge\delta$.
    This contradicts the choice of $g$.
\end{proof}

\begin{rem}
    The preceding proposition holds for a strictly outer cocycle action $(\alpha,u)\colon G\curvearrowright R$ with the same proof, because it is perturbed to a genuine action on $K$ (\cite[Theorem 4.5]{Ni26}).
\end{rem}

We observe that $\mathbb{C}1\wtensor R_{\omega}\subset (M\wtensor R)_{\omega}$ canonically and the same argument holds for equicontinuous parts.
Therefore we obtain the permanence property for the central freeness of actions on McDuff factors.

\begin{cor}
    Let $G$ be a locally compact group with a compact open subgroup.
    If $\alpha\colon G\curvearrowright R$ is strictly outer and $\beta\colon G\curvearrowright M$ is an arbitrary action, then the diagonal action $\beta\otimes\alpha\colon G\curvearrowright M\wtensor R$ is centrally free.
\end{cor}

\subsection{Rohlin tower construction}

Throughout this subsection, $(\alpha,u)\colon G\curvearrowright M$ denotes a centrally free cocycle action of an amenable locally compact group $G$ with a compact open subgroup $K<G$, and we put $N=M_{\omega,\alpha}$.
We apply Popa's local quantization principle to the inclusion $N^Kp_K=p_K(N\rtimes K)p_K\subset p_K(N\rtimes G)p_K$.
The following formulation follows from the argument of \cite{Po82}.

\begin{prop}
    Let $B\subset A$ be an inclusion of von Neumann algebras with expectation $\mathbb{E}$.
    Let $\tau$ be a faithful normal tracial state on $B$ and set $\varphi=\tau\circ\mathbb{E}$.
    Suppose that $B'\cap A\subset B$.
    For every $\varepsilon>0$ and $x_1,\ldots,x_n\in A\ominus B=\ker{\mathbb{E}}$, there exists a finite partition of unity $(e_j)_{j=1}^m$ in $B$ such that
    \[\sum_j\|e_jx_ie_j\|_{\varphi}^2\le\varepsilon\|x_i\|_{\varphi}^2\]
    for all $i=1,\ldots,n$.
\end{prop}

Let $\tau$ denote the canonical trace $\tau_{\omega}$ on $N$ and also the trace on $N^Kp_K$ determined by $\tau(xp_K)=\tau_{\omega}(x)$ for $x\in N^K$.
This trace is invariant under $\alpha$.
The restriction of the conditional expectation $\mathbb{E}_K\colon N\rtimes G\to N\rtimes K$ gives a faithful normal expectation $\mathbb{E}\colon p_K(N\rtimes G)p_K\to N^Kp_K$.
Set a state $\varphi=\tau\circ\mathbb{E}$ on $p_K(N\rtimes G)p_K$.
Equivalently, if $\hat{\tau}$ denotes the dual weight of $\tau$ on $N\rtimes G$, then $\varphi$ is the restriction $\varphi=\hat{\tau}p_K$ on $p_K(N\rtimes G)p_K$.
Then $\varphi(p_Kxp_K)=\tau(x)$ for every $x\in N$.

For every $s\in (G/K)\setminus\{K\}$ and a projection $e\in N^K$, a direct calculation shows that $\mathbb{E}(p_K\lambda_s^{\ast}p_K\lambda_sp_K)=|K\cap sKs^{-1}|p_K$ and
\begin{align*}
    \|ep_K\lambda_sp_Ke\|_{\varphi}^2 &= \varphi(ep_K\lambda_s^{\ast}ep_K\lambda_sep_K)\\
    &= \varphi(e\mathbb{E}(\lambda_s^{\ast}ep_K\lambda_s)p_K)\\
    &= \varphi(e\alpha_s^{-1}(e)\mathbb{E}(\lambda_s^{\ast}p_K\lambda_s)p_K)\\
    &= |K\cap sKs^{-1}|\varphi(e\alpha_s^{-1}(e)p_K)\\
    &= |K\cap sKs^{-1}|\tau(e\alpha_s(e)).
\end{align*}

Let $S_0\subset (G/K)\setminus\{K\}$ be a nonempty finite subset.
There exists a partition of unity $(e_j)_j$ in $N^K$ such that
\[\sum_j\|e_jp_K\lambda_sp_Ke_j\|_{\varphi}^2\le\varepsilon\|p_K\lambda_sp_K\|_{\varphi}^2\]
for every $s\in S_0$.
This implies that $\sum_j\tau(e_j\alpha_s(e_j))\le\varepsilon=\sum_j\varepsilon\tau(e_j)$ for every $s\in S_0$.
Summing over $s\in S_0$, we find a nonzero projection $e=e_j\in N^K$ such that
\[\tau(e\alpha_s(e))\le |S_0|\varepsilon\tau(e).\]
The same argument applies to any corner with a projection in $N^K$.

\begin{lem}
    Let $S_0\subset (G/K)\setminus\{K\}$ be a finite $K$-invariant subset.
    Then there exists a nonzero projection $e\in N^K$ such that $e\alpha_s(e)=0$ for every $s\in S_0$.
\end{lem}
\begin{proof}
    By enlarging $S_0$, we may assume that $S_0$ is symmetric in the sense that $(\bigcup S_0)^{-1}=\bigcup S_0\subset G$.
    Let $\delta>0$ and $e\in N^K$ a maximal projection satisfying $\tau(e\alpha_s(e))\le\delta\tau(e)$ for every $s\in S_0$.
    Suppose that $1\neq f=e\vee\bigvee_{s\in S_0}\alpha_s(e)\in N^K$.
    Applying the local quantization to the corner of $1-f$, we obtain a nonzero projection $e'\in N^K$ such that $e'\le1-f$ and $\tau(e'\alpha_s(e'))\le\delta\tau(e')$ for every $s\in S_0$.
    Since $S_0$ is symmetric, we have $e'\alpha_s(e)=0=e\alpha_s(e')$ for every $s\in S_0$.
    The projection $e+e'$ satisfies the same inequality as $e$, which contradicts the maximality of $e$.
    Therefore, $f=1$ and
    \[1=\tau(f)\le\tau(e)+\sum_{s\in S_0}\tau(\alpha_s(e))=(1+|S_0|)\tau(e).\]

    Using the index selection trick as $\delta\searrow0$, we obtain a projection $e\in N^K$ with $\tau(e)\ge(1+|S_0|)^{-1}>0$ and $\tau(e\alpha_s(e))=0$ for every $s\in S_0$.
    Since $\tau$ is faithful, this means $e\alpha_s(e)=0$.
\end{proof}

Let $F=FK\subset G$ be compact and $\varepsilon>0$, and fix a finite $K$-$(F,\varepsilon)$-invariant subset $S\subset G/K$.
Let $\mathcal{E}$ be the set of all families $E=(e_s)_{s\in S}$ of mutually orthogonal projections in $N^{K_S}$ such that
\begin{itemize}
    \item $\alpha_k(e_s)=e_{ks}$ for all $k\in K$ and $s\in S$;
    \item $[\alpha_g(e_s),e_{s'}]=0$ for all $g\in F$ and $s,s'\in S$.
\end{itemize}
We adopt the convention that $e_s=0$ for $s\notin S$.
For $E\in\mathcal{E}$ and $g\in F$, define
\[a_{g,E}=\sum_{s\in G/K}|\alpha_g(e_s)-e_{gs}|_1,\quad b_E=\sum_s|e_s|_1.\]

\begin{lem}\label{lem:E-to-E'}
    Let $E\in\mathcal{E}$.
    If $b_E<1-\varepsilon^{1/2}$, then there exists $E'\in \mathcal{E}$ such that
    \begin{enumerate}
        \item $0<(\varepsilon^{1/2}/2)\sum_s|e'_s-e_s|_1\le b_{E'}-b_E$;
        \item $a_{g,E'}-a_{g,E}\le 2\varepsilon^{1/2}(b_{E'}-b_E)$ for all $g\in F$.
    \end{enumerate}
\end{lem}
\begin{proof}
    Let $f\in N^K$ be a nonzero projection such that $f\alpha_s(f)=0$ for all $s\in S'$, where $S'=(\bigcup FS)^{-1}FS\setminus\{K\}\subset G/K$.
    By the fast reindexation, the projection $f$ can be chosen to be $\tau$-orthogonal to $(\alpha_g(e_s))_{g\in G,s\in S}''$.
    Then the projections $\alpha_s(f)$, $s\in FS$, are mutually orthogonal.
    Let $\tilde{f}=\sum_{s\in S}\alpha_s(f)$ and $e_s'=e_s(1-\tilde{f})+\alpha_s(f)$.
    Then $E'=(e'_s)_{s\in S}$ belongs to $\mathcal{E}$.

    To see (1), we calculate
    \begin{align*}
        \sum_s|e_s'-e_s|_1 &\le \sum_s(|e_s\tilde{f}|_1+|\alpha_s(f)|_1)\\
        &= |(\sum_se_s)\tilde{f}|_1+|S||f|_1\\
        &< ((1-\varepsilon^{1/2})+1)|S||f|_1\le 2|S||f|_1,
    \end{align*}
    and
    \begin{align*}
        b_{E'}-b_E &= \sum_s(-|e_s\tilde{f}|_1+|\alpha_s(f)|_1)\\
        &> -(1-\varepsilon^{1/2})|S||f|_1+|S||f|_1\\
        &= \varepsilon^{1/2}|S||f|_1.
    \end{align*}
    These inequalities prove $0<(\varepsilon^{1/2}/2)\sum_s|e_s'-e_s|_1<\varepsilon^{1/2}|S||f|_1<b_{E'}-b_E$.

    Let $g\in F$.
    For $s\in S\cap g^{-1}S$, we have
    \begin{align*}
        |\alpha_g(e_s')-e_{gs}'|_1 &= |\alpha_g(e_s)(1-\alpha_g(\tilde{f}))-e_{gs}(1-\tilde{f})|_1\\
        &\le |(\alpha_g(e_s)-e_{gs})(1-\tilde{f})|_1+|\alpha_g(e_s)(\tilde{f}-\alpha_g(\tilde{f}))|_1\\
        &\le |\alpha_g(e_s)-e_{gs}|_1+|\alpha_g(e_s)(\tilde{f}-\alpha_g(\tilde{f}))|_1.
    \end{align*}
    For $s\in S\setminus g^{-1}S$,
    \begin{align*}
        |\alpha_g(e_s')-e_{gs}'|_1 &= |\alpha_g(e_s')|_1\\
        &= |e_s(1-\tilde{f})|_1+|\alpha_s(f)|_1\\
        &\le |e_s|_1+|f|_1\\
        &= |\alpha_g(e_s)-e_{gs}|_1+|f|_1.
    \end{align*}
    Similarly, for $s\in g^{-1}S\setminus S$, we have $|\alpha_g(e_s')-e_{gs}'|_1\le|\alpha_g(e_s)-e_{gs}|_1+|f|_1$.
    Note that $|\tilde{f}-\alpha_g(\tilde{f})|_1=|S\Delta gS||f|_1\le\varepsilon|S||f|_1$ by the F\o lner condition.
    Thus, it follows that
    \begin{align*}
        a_{g,E'}-a_{g,E} &\le \sum_{s\in S\cap g^{-1}S}|\alpha_g(e_s)(\tilde{f}-\alpha_g(\tilde{f}))|_1+\sum_{s\in S\Delta g^{-1}S}|f|_1\\
        &\le |\tilde{f}-\alpha_g(\tilde{f})|_1+|S\Delta g^{-1}S||f|_1\\
        &\le 2\varepsilon|S||f|_1\\
        &= 2\varepsilon^{1/2}\cdot \varepsilon^{1/2}|S||f|_1\\
        &< 2\varepsilon^{1/2}(b_{E'}-b_E).
    \end{align*}
    This proves (2).
\end{proof}

\begin{thm}\label{thm:tower}
    Let $(\alpha,u)\colon G\curvearrowright M$ be a centrally free cocycle action of an amenable locally compact group with a compact open subgroup $K<G$ on a factor $M$, and let $S\subset G/K$ be a finite $K$-$(F,\varepsilon)$-invariant subset.
    There exists a partition of unity $E=(e_s)_{s\in S}$ in $(M_{\omega,\alpha})^{K_S}$ such that $\alpha_k(e_s)=e_{ks}$ for every $k\in K$, $s\in S$ and
    \[a_{g,E}=\sum_{s\in G/K}|\alpha_g(e_s)-e_{gs}|_1\le 4\varepsilon^{1/2}\]
    for every $g\in F$.
\end{thm}
\begin{proof}
    Let $\mathcal{E}_0$ be the set of all $E\in\mathcal{E}$ such that $a_{g,E}\le 2\varepsilon^{1/2}b_E$ for every $g\in F$.
    This set contains the zero family and hence is nonempty.
    We define an order $E\le E'$ on $\mathcal{E}_0$ if $E=E'$ or $E$ and $E'$ satisfy the conditions (1) and (2) in Lemma \ref{lem:E-to-E'}.
    Then $\mathcal{E}_0$ is an inductive ordered set.
    By Zorn's lemma, we find a maximal element $E'=(e'_s)_{s\in S}\in\mathcal{E}_0$.
    Lemma \ref{lem:E-to-E'} implies that $b_{E'}\ge 1-\varepsilon^{1/2}$.
    Define $E=(e_s)_{s\in S}$ by $e_s=e_s'$ for $s\neq K$ and $e_K=1-\sum_{s\neq K}e_s'\ge e_K'$.
    Then the conclusion follows.
\end{proof}

Note that $\sum_{s\in S\setminus gS}|e_s|_1=\sum_{s\in S\setminus gS}|\alpha_g(e_{g^{-1}s})-e_s|_1\le a_{g,E}\le 4\varepsilon^{1/2}$ for $g\in F$.

\section{Cohomology of actions}\label{sec:3}

\subsection{\texorpdfstring{$K$}{K}-invariant cocycles and continuity}

Let $G$ be a locally compact group with a compact open subgroup $K<G$, and let $A$ be a Polish group.
We say that a map $v\colon G\to A$ is \emph{(right) $K$-invariant} if $v_{gk}=v_g$ for all $g\in G$ and $k\in K$, or equivalently, it descends
canonically to a map $G/K\to A$.
For an $\alpha$-cocycle $v$, it is $K$-invariant if and only if $v|_K\equiv 1$ if and only if it is (left) $K$-equivariant, that is, $v_{kg}=\alpha_k(v_g)$ for all $g\in G$ and $k\in K$.
If $v|_K$ is a coboundary, then $v$ is cohomologous to a $K$-invariant one.

Recall that every measurable group homomorphism between Polish groups is continuous.
This fact implies the following standard fact.

\begin{lem}
    Every Borel 1-cocycle is continuous.
\end{lem}

We say that a 2-cocycle $u\colon G\times G\to A$ (together with $\alpha\colon G\to\Aut(A)$) is \emph{(right) $K$-invariant} if $u_{g,k}=1$ for all $g\in G$ and $k\in K$, or equivalently, it descends to $G\times G/K\to A$.
Note that if a cocycle action $(\alpha,u)$ is $K$-invariant, then $\alpha_g\circ\alpha_k=\alpha_{gk}$ for every $g\in G$ and $k\in K$.
In particular, $\alpha$ is continuous.

\begin{lem}
    If $u|_K$ is a coboundary, then $u$ is perturbed to a $K$-invariant one.
\end{lem}
\begin{proof}
    There exists a Borel map $v\colon K\to A$ such that $v_k\alpha_k(v_l)u_{k,l}v_{kl}^{-1}=1$ for all $k,l\in K$.
    Extend it to a 1-cochain $v\colon G\to A$ by setting
    \[v_g=\alpha_{\sigma(gK)}(v_{\sigma(gK)^{-1}g})u_{\sigma(gK),\sigma(gK)^{-1}g}.\]
    The perturbed cocycle $u'$, defined by $u'_{g,h}=v_g\alpha_g(v_h)u_{g,h}v_{gh}^{-1}$, is $K$-invariant.
    For notational convenience, we write $g=sl$ where $s=\sigma(gK)$ and $l=\sigma(gK)^{-1}g\in K$.
    For $k\in K$, we calculate
    \begin{align*}
        u'_{g,k} &= v_g\alpha_g(v_k)u_{g,k}v_{gk}^{-1}\\
        &= \alpha_s(v_l)u_{s,l}\cdot\alpha_{sl}(v_k)\cdot u_{sl,k}\cdot u_{s,lk}^{-1}\alpha_s(v_{lk}^{-1})\\
        &= \alpha_s(v_l)\alpha_s(\alpha_l(v_k))\cdot u_{s,l}u_{sl,k}u_{s,lk}^{-1}\cdot\alpha_s(v_{lk}^{-1})\\
        &= \alpha_s(v_l\alpha_l(v_k)u_{l,k}v_{lk}^{-1})=1.
    \end{align*}
    Hence $u'$ is $K$-invariant.
\end{proof}

\begin{lem}
    If a 2-cocycle $u$ is $K$-invariant, then $u$ is continuous.
\end{lem}
\begin{proof}
    Obviously $u$ is locally constant in the second variable.
    It suffices to show that $v\colon g\mapsto u_{g,s}^{-1}$ is continuous for each fixed $s\in G/K$.
    Note that the subgroup $K_s=K\cap sKs^{-1}$ is open, and the restriction $\alpha|_{K_s}$ is an action.
    For every $k\in K_s$, we have $ks=s$ and
    \[u_{gk,s}=\alpha_g(u_{k,s})u_{g,ks}\]
    by the 2-cocycle identity.
    Consequently, $v_{gk}=v_g\alpha_g(v_k)$ holds for all $g\in G$ and $k\in K_s$, which means that $v|_{K_s}$ is an $\alpha$-cocycle and therefore is continuous.
\end{proof}

\subsection{Cohomology vanishing}

Fix a faithful normal state $\varphi$ on $M$.
The next lemma is the main averaging device for the cohomology arguments.
It combines the Rohlin tower from Theorem \ref{thm:tower} with $K$-invariance to make a cochain approximately trivial, with the remaining error controlled by a translated defect of the original cochain.

\begin{lem}\label{lem:approx-1coh-vanish}
    Let $(\alpha,u)\colon G\curvearrowright M$ be a centrally free cocycle action such that $u$ is $K$-invariant.
    Let $F=FK\subset G$ be compact, $\varepsilon>0$, and $S\subset G/K$ a finite $K$-$(F,\varepsilon)$-invariant subset.
    Let $V\colon G\to\U(M^{\omega}_{\alpha})$ be a $K$-invariant map.
    Then there exists $W\in\U((M^K)^{\omega})$ such that
    \begin{enumerate}
        \item $|a(WV_g\alpha_g(W^{\ast})-1)b|_{\varphi^{\omega}}^{\sharp}\le 4\varepsilon^{1/2}+\sup_{h\in\bigcup S}|a(\alpha_{h^{-1}}^{-1}(V_{h^{-1}})V_g\alpha_g(\alpha_{h^{-1}g}^{-1}(V_{h^{-1}g}^{\ast}))-1)b|_{\varphi^{\omega}}^{\sharp}$;
        \item $\|[WV_g\alpha_g(W^{\ast}),\psi^{\omega}]\|\le 8\varepsilon^{1/2}+\sup_{h\in\bigcup S}\|[\alpha_{h^{-1}}^{-1}(V_{h^{-1}})V_g\alpha_g(\alpha_{h^{-1}g}^{-1}(V_{h^{-1}g}^{\ast})),\psi^{\omega}]\|$;
        \item $\|[W,\psi^{\omega}]\|\le \sup_{h\in\bigcup S}\|[V_{h^{-1}},\alpha_{h^{-1}}(\psi^{\omega})]\|$
    \end{enumerate}
    for all $g\in F$, $a,b\in \U(M)$ and $\psi\in\nS(M)$.
    If $V$ takes values in $\U(M_{\omega,\alpha})$, then $W$ can be chosen in $\U((M_{\omega,\alpha})^K)$.
\end{lem}
\begin{proof}
    Let $Q$ be a countably generated, globally $\sigma^{\varphi^{\omega}}$- and $\alpha$-invariant von Neumann subalgebra of $M^{\omega}_{\alpha}$ containing $M$ and $V_g$, $g\in G$.
    By the Rohlin property of $\alpha|_K$ and the fast reindexation, we regard $L^{\infty}(K)\wtensor Q\hookrightarrow M^{\omega}_{\alpha}$ as a $K$-equivariant subalgebra where $K$ acts diagonally.

    Let $(e_s)_{s\in S}$ be a Rohlin tower obtained by Theorem \ref{thm:tower} that is orthogonal to $L^{\infty}(K)\wtensor Q$ in $M^{\omega}_{\alpha}$.
    For each $s\in G/K$, define $W_s\in L^{\infty}(K)\wtensor Q$ by $W_s(k)=\alpha_{\sigma(s)}^{-1}(\alpha_{k^{-1}\sigma(s)^{-1}}^{-1}(V_{k^{-1}\sigma(s)^{-1}}))$ and set $W=\sum_s\alpha_{\sigma(s)}(W_s)e_s$.
    Observe that $W$ is a unitary.

    Using the cocycle identities and the $K$-invariance of $u$ and $V$, a direct calculation shows $\alpha_k(\alpha_{\sigma(s)}(W_s))=\alpha_{\sigma(ks)}(W_{ks})$ for $k\in K$ and $s\in S$.
    Indeed, we have
    \[\alpha_k(\alpha_{\sigma(s)}(W_s))=\alpha_{\sigma(ks)}(\alpha_{\sigma(ks)}^{-1}(u_{k,s})\alpha_{\sigma(ks)^{-1}k\sigma(s)}(W_s)\alpha_{\sigma(ks)}^{-1}(u_{k,s}^{\ast})),\]
    and
    \begin{align*}
        &\mathrel{\phantom{=}} [\alpha_{\sigma(ks)}^{-1}(u_{k,s})\alpha_{\sigma(ks)^{-1}k\sigma(s)}(W_s)\alpha_{\sigma(ks)}^{-1}(u_{k,s}^{\ast})](l)\\
        &= \alpha_{\sigma(ks)}^{-1}(u_{k,s})\cdot\alpha_{\sigma(ks)^{-1}k\sigma(s)}(W_s(\sigma(s)^{-1}k^{-1}\sigma(ks)l))\cdot\alpha_{\sigma(ks)}^{-1}(u_{k,s}^{\ast})\\
        &= \alpha_{\sigma(ks)}^{-1}(u_{k,s})\cdot\alpha_{\sigma(ks)^{-1}k\sigma(s)}(\alpha_{\sigma(s)}^{-1}(\alpha_{l^{-1}\sigma(ks)^{-1}k}^{-1}(V_{l^{-1}\sigma(ks)^{-1}k})))\cdot\alpha_{\sigma(ks)}^{-1}(u_{k,s}^{\ast})\\
        &= \alpha_{\sigma(ks)}^{-1}(u_{k,s})\cdot\alpha_{k^{-1}\sigma(ks)}^{-1}(\alpha_{l^{-1}\sigma(ks)^{-1}k}^{-1}(V_{l^{-1}\sigma(ks)^{-1}k}))\cdot\alpha_{\sigma(ks)}^{-1}(u_{k,s}^{\ast})\\
        &= \alpha_{\sigma(ks)}^{-1}(u_{k,s})\cdot\alpha_{k^{-1}\sigma(ks)}^{-1}(\alpha_k^{-1}(\alpha_{l^{-1}\sigma(ks)^{-1}}^{-1}(V_{l^{-1}\sigma(ks)^{-1}})))\cdot\alpha_{\sigma(ks)}^{-1}(u_{k,s}^{\ast})\\
        &= \alpha_{\sigma(ks)}^{-1}(\alpha_{l^{-1}\sigma(ks)^{-1}}^{-1}(V_{l^{-1}\sigma(ks)^{-1}}))\\
        &= W_{ks}(l).
    \end{align*}
    Since $\alpha_k(e_s)=e_{ks}$, it follows that $\alpha_k(W)=W$, $W\in \U((M^K)^{\omega})$.

    We next compute the coefficients arising from the tower.
    For $g\in F$, $s\in S$ and $k\in K$, put $h=\sigma(gs)k\in\bigcup gS$.
    We see that
    \[[W_{gs}\alpha_{\sigma(gs)}^{-1}(V_g)\alpha_{\sigma(gs)}^{-1}\circ\alpha_g\circ\alpha_{\sigma(s)}(W_s^{\ast})](k)=\alpha_{\sigma(gs)}^{-1}(\alpha_{h^{-1}}^{-1}(V_{h^{-1}})V_g\alpha_g(\alpha_{h^{-1}g}^{-1}(V_{h^{-1}g}^{\ast})))\]
    by the calculation
    \begin{align*}
        &\mathrel{\phantom{=}} [W_{gs}\alpha_{\sigma(gs)}^{-1}(V_g)\alpha_{\sigma(gs)}^{-1}\circ\alpha_g\circ\alpha_{\sigma(s)}(W_s^{\ast})](k)\\
        &= [W_{gs}\alpha_{\sigma(gs)}^{-1}(V_g)\alpha_{\sigma(gs)}^{-1}(u_{g,s})\alpha_{\sigma(gs)^{-1}g\sigma(s)}(W_s^{\ast})\alpha_{\sigma(gs)}^{-1}(u_{g,s}^{\ast})](k)\\
        &= W_{gs}(k)\cdot\alpha_{\sigma(gs)}^{-1}(V_g)\cdot\alpha_{\sigma(gs)}^{-1}(u_{g,s})\cdot\alpha_{\sigma(gs)^{-1}g\sigma(s)}(W_s(\sigma(s)^{-1}g^{-1}\sigma(gs)k)^{\ast})\cdot\alpha_{\sigma(gs)}^{-1}(u_{g,s}^{\ast})\\
        &= \alpha_{\sigma(gs)}^{-1}(\alpha_{k^{-1}\sigma(gs)^{-1}}^{-1}(V_{k^{-1}\sigma(gs)^{-1}}))\cdot\alpha_{\sigma(gs)}^{-1}(V_g)\cdot\alpha_{\sigma(gs)}^{-1}(u_{g,s})\\
        &\quad\ \cdot\alpha_{\sigma(gs)^{-1}g\sigma(s)}(\alpha_{\sigma(s)}^{-1}(\alpha_{k^{-1}\sigma(gs)^{-1}g}^{-1}(V_{k^{-1}\sigma(gs)^{-1}g}^{\ast})))\cdot\alpha_{\sigma(gs)}^{-1}(u_{g,s}^{\ast})\\
        &= \alpha_{\sigma(gs)}^{-1}(\alpha_{k^{-1}\sigma(gs)^{-1}}^{-1}(V_{k^{-1}\sigma(gs)^{-1}})V_g\alpha_g(\alpha_{k^{-1}\sigma(gs)^{-1}g}^{-1}(V_{k^{-1}\sigma(gs)^{-1}g}^{\ast}))).
    \end{align*}
    We decompose the term $WV_g\alpha_g(W^{\ast})-1$ as
    \begin{align*}
        WV_g\alpha_g(W^{\ast})-1 &= WV_g\sum_s\alpha_g(\alpha_{\sigma(s)}(W_s^{\ast}))(\alpha_g(e_s)-e_{gs})\\
        &\quad\ +\sum_s\alpha_{\sigma(gs)}(W_{gs}\alpha_{\sigma(gs)}^{-1}(V_g)\alpha_{\sigma(gs)}^{-1}\circ\alpha_g\circ\alpha_{\sigma(s)}(W_s^{\ast})-1)e_{gs}.
    \end{align*}
    The item (1) now follows as
    \begin{align*}
        &\quad\ |a(WV_g\alpha_g(W^{\ast})-1)b|_{\varphi^{\omega}}^{\sharp}\\ &\le \sum_s|\alpha_g(e_s)-e_{gs}|_1+\max_s\sup_{h\in\bigcup S}|\alpha_{\sigma(gs)}^{-1}(a(\alpha_{h^{-1}}^{-1}(V_{h^{-1}})V_g\alpha_g(\alpha_{h^{-1}g}^{-1}(V_{h^{-1}g}^{\ast}))-1)b)|_{\varphi^{\omega}\circ\alpha_{\sigma(gs)}}^{\sharp}\\
        &\le 4\varepsilon^{1/2}+\sup_{h\in\bigcup S}|a(\alpha_{h^{-1}}^{-1}(V_{h^{-1}})V_g\alpha_g(\alpha_{h^{-1}g}^{-1}(V_{h^{-1}g}^{\ast}))-1)b|_{\varphi^{\omega}}^{\sharp}.
    \end{align*}
    For the commutator estimate, we use $\|[x,\psi]\|\le 2|x|_{\psi}^{\sharp}$ and apply the same calculation.
    The formula defining $W$ provides (3) and the last assertion.
\end{proof}

\begin{rem}
    For a (normalized) 1-cochain $V$, the defect appearing in the preceding lemma is expanded as
    \begin{align*}
        &\mathrel{\phantom{=}}\alpha_{h^{-1}}^{-1}(V_{h^{-1}})V_g\alpha_g(\alpha_{h^{-1}g}^{-1}(V_{h^{-1}g}^{\ast}))\\
        &= u_{h,h^{-1}}^{\ast}\alpha_h(V_{h^{-1}})u_{h,h^{-1}}\cdot V_g\cdot u_{h,h^{-1}g}^{\ast}\alpha_h(V_{h^{-1}g}^{\ast})u_{h,h^{-1}g}\\
        &= u_{h,h^{-1}}^{\ast}V_h^{\ast}[V_h\alpha_h(V_{h^{-1}})u_{h,h^{-1}}-1]V_gu_{h,h^{-1}g}^{\ast}\alpha_h(V_{h^{-1}g}^{\ast})u_{h,h^{-1}g}\\
        &\quad\ +u_{h,h^{-1}}^{\ast}[V_h^{\ast}V_gu_{h,h^{-1}g}^{\ast}\alpha_h(V_{h^{-1}g}^{\ast})-1]u_{h,h^{-1}g}\\
        &\quad\ +u_{h,h^{-1}}^{\ast}u_{h,h^{-1}g}.
    \end{align*}
    In particular, if $u=\partial V$, only the last term $u_{h,h^{-1}}^{\ast}u_{h,h^{-1}g}$ remains.
\end{rem}

\begin{thm}[1-cohomology vanishing]\label{thm:1-coh-vanish}
    Let $\alpha\colon G\curvearrowright M$ be a centrally free action of an amenable locally compact group $G$ with a compact open subgroup $K$ on a factor $M$.
    Then every $\alpha$-cocycle $G\to \U(M_{\omega,\alpha})$ is a coboundary in $M_{\omega,\alpha}$.
\end{thm}
\begin{proof}
    On the compact group $K$, this is already known; see the proof of \cite[Theorem 4.1]{Ni26}.
    Thus, we may assume $V$ is $K$-invariant.
    Lemma \ref{lem:approx-1coh-vanish} implies that $V$ is approximated by $W^{\ast}\alpha_g(W)$ for some $W\in\U((M_{\omega,\alpha})^K)$.
    Applying the index selection trick, we obtain the conclusion.
\end{proof}

\begin{lem}\label{lem:approx-2coh-vanish}
    Let $(\alpha,u)\colon G\curvearrowright M$ be a $K$-invariant centrally free cocycle action on a \emph{finite} factor $M$.
    Suppose that the following data are given:
    \begin{itemize}
        \item a finite $K$-$(F,\varepsilon)$-invariant subset $S\subset G/K$;
        \item a number $\delta>0$;
        \item a compact subset $E=EK\subset G$.
    \end{itemize}
    Then there exists a $K$-invariant map $v\colon G\to\U(M)$ such that
    \begin{enumerate}[label=\textup{(\alph*)}]
        \item $|v_g\alpha_g(v_h)u_{g,h}v_{gh}^{\ast}-1|_1<\delta$;
        \item $|v_p-1|_1<4\varepsilon^{1/2}+\sup_{h\in \bigcup S}3|u_{h^{-1},p}-1|_1+\delta$
    \end{enumerate}
    for all $g,h\in E$ and $p\in F$.
\end{lem}
\begin{proof}
    We may assume that $E$ contains $F$.
    Choose $\varepsilon'>0$, to be specified below, and let $S'\subset G/K$ be a finite $K$-$(E,\varepsilon')$-invariant subset.
    Let $(e_s)_{s\in S'}$ be a Rohlin tower of $S'$, and define $V^{(0)}\colon G\to\U(M^{\omega}_{\alpha})$ by $V_g^{(0)}=\sum_su_{g,s}^{\ast}\alpha_g(e_s)$.
    Then, for all $g\in G$ and $k\in K$, it follows that
    \begin{align*}
        V_g^{(0)}\alpha_g(V_k^{(0)}) &= (\sum_su_{g,ks}^{\ast}\alpha_g(e_{ks}))(\sum_s\alpha_g(u_{k,s}^{\ast})\alpha_{gk}(e_s))\\
        &= \sum_su_{g,ks}^{\ast}\alpha_g(u_{k,s}^{\ast})\alpha_{gk}(e_s)\\
        &= \sum_su_{gk,s}^{\ast}\alpha_{gk}(e_s)=V_{gk}^{(0)}.
    \end{align*}
    We also have that
    \begin{align*}
        |V_g^{(0)}\alpha_g(V_h^{(0)})u_{g,h}V_{gh}^{(0)\ast}-1|_1 &\le |(\sum_su_{g,hs}^{\ast}(\alpha_g(e_{hs})-\alpha_{gh}(e_s)))\alpha_g(V_h^{(0)})u_{g,h}V_{gh}^{(0)\ast}|_1\\
        &\quad\ + |\sum_s(u_{g,hs}^{\ast}\alpha_g(u_{h,s}^{\ast})u_{g,h}u_{gh,s}-1)\alpha_{gh}(e_s)|_1\\
        &\le \sum_s|e_{hs}-\alpha_h(e_s)|_1+0\le 4\varepsilon'^{1/2}
    \end{align*}
    for $g\in G$ and $h\in E$.
    Moreover,
    \begin{align*}
        &\quad\ |\alpha_{h^{-1}}^{-1}(V_{h^{-1}}^{(0)})V_g^{(0)}\alpha_g(\alpha_{h^{-1}g}^{-1}(V_{h^{-1}g}^{(0)\ast}))-1|_1\\
        &\le |\sum_s\alpha_{h^{-1}}^{-1}(u_{h^{-1},gs}^{\ast})\alpha_{h^{-1}}^{-1}(\alpha_{h^{-1}}(e_{gs}-\alpha_g(e_s)))|_1\\
        &\quad\ +|\sum_s(\alpha_{h^{-1}}^{-1}(u_{h^{-1},gs}^{\ast})u_{g,s}^{\ast}\alpha_g(\alpha_{h^{-1}g}^{-1}(u_{h^{-1}g,s}))-1)\cdot\alpha_g(e_s)|_1\\
        &\le \sum_s|e_{gs}-\alpha_g(e_s)|_1+|\sum_s(\alpha_{h^{-1}}^{-1}(u_{h^{-1},gs}^{\ast}\alpha_{h^{-1}}(u_{g,s}^{\ast})u_{h^{-1},g}u_{h^{-1}g,s}u_{h^{-1},g}^{\ast})-1)\alpha_g(e_s)|_1\\
        &\le 4\varepsilon'^{1/2}+|\alpha_{h^{-1}}^{-1}(u_{h^{-1},g}^{\ast}-1)|_1\\
        &= 4\varepsilon'^{1/2}+|u_{h^{-1},g}-1|_1
    \end{align*}
    for $g\in E$.
    
    Next, using the Rohlin property of $\alpha|_K$, we define $W^{(0)}\in\U(M^{\omega}_{\alpha})$ by the function $K\ni k\mapsto V_k^{(0)}$.
    This unitary satisfies $W^{(0)}\alpha_k(W^{(0)\ast})=V_k^{(0)}$ for every $k\in K$.
    In particular, we have
    \begin{align*}
        V_g^{(0)} &= V_{\sigma(gK)}^{(0)}\alpha_{\sigma(gK)}(V_{\sigma(gK)^{-1}g}^{(0)})\\
        &= V_{\sigma(gK)}^{(0)}\alpha_{\sigma(gK)}(W^{(0)}\alpha_{\sigma(gK)^{-1}g}(W^{(0)\ast}))\\
        &= V_{\sigma(gK)}^{(0)}\alpha_{\sigma(gK)}(W^{(0)})\alpha_g(W^{(0)\ast}).
    \end{align*}
    Define $V^{(1)}_g=W^{(0)\ast}V_{\sigma(gK)}^{(0)}\alpha_{\sigma(gK)}(W^{(0)})=W^{(0)\ast}V_g^{(0)}\alpha_g(W^{(0)})$, which is right $K$-invariant.
    Then
    \begin{align*}
        V_g^{(1)}\alpha_g(V_h^{(1)})u_{g,h}V_{gh}^{(1)\ast} &= W^{(0)\ast}V_g^{(0)}\alpha_g(W^{(0)})\cdot\alpha_g(W^{(0)\ast})\alpha_g(V_h^{(0)})\alpha_g(\alpha_h(W^{(0)}))\\
        &\quad\ \cdot u_{g,h}\cdot \alpha_{gh}(W^{(0)\ast})V_{gh}^{(0)\ast}W^{(0)}\\
        &= W^{(0)\ast}\cdot V_g^{(0)}\alpha_g(V_h^{(0)})u_{g,h}V_{gh}^{(0)\ast}\cdot W^{(0)}
    \end{align*}
    and, combining this with the preceding estimate, we obtain that
    \[|V_g^{(1)}\alpha_g(V_h^{(1)})u_{g,h}V_{gh}^{(1)\ast}-1|_1\le 4\varepsilon'^{1/2}\]
    whenever $h\in E$.
    On the other hand, the cocycle relation gives
    \begin{align*}
        |\alpha_{h^{-1}}^{-1}(W^{(0)\ast})\alpha_g(\alpha_{h^{-1}g}^{-1}(W^{(0)}))-1|_1 &= |\alpha_g(\alpha_{h^{-1}g}^{-1}(u_{h^{-1},g}^{\ast}W^{(0)\ast}u_{h^{-1},g}W^{(0)}-1))|_1\\
        &\le 2|u_{h^{-1},g}-1|_1.
    \end{align*}
    Therefore, it follows that
    \begin{align*}
        &\quad\ |\alpha_{h^{-1}}^{-1}(V_{h^{-1}}^{(1)})V_g^{(1)}\alpha_g(\alpha_{h^{-1}g}^{-1}(V_{h^{-1}g}^{(1)\ast}))-1|_1\\
        &= |\alpha_{h^{-1}}^{-1}(W^{(0)\ast}V_{h^{-1}}^{(0)}\alpha_{h^{-1}}(W^{(0)}))\cdot W^{(0)\ast}V_g^{(0)}\alpha_g(W^{(0)})\cdot\alpha_g(\alpha_{h^{-1}g}^{-1}(\alpha_{h^{-1}g}(W^{(0)\ast})V_{h^{-1}g}^{(0)\ast}W^{(0)}))-1|_1\\
        &= |\alpha_{h^{-1}}^{-1}(W^{(0)\ast})\cdot\alpha_{h^{-1}}^{-1}(V_{h^{-1}}^{(0)})V_g^{(0)}\alpha_g(\alpha_{h^{-1}g}^{-1}(V_{h^{-1}g}^{(0)\ast}))\cdot\alpha_g(\alpha_{h^{-1}g}^{-1}(W^{(0)}))-1|_1\\
        &\le |\alpha_{h^{-1}}^{-1}(V_{h^{-1}}^{(0)})V_g^{(0)}\alpha_g(\alpha_{h^{-1}g}^{-1}(V_{h^{-1}g}^{(0)\ast}))-1|_1+2|u_{h^{-1},g}-1|_1\\
        &\le 4\varepsilon'^{1/2}+3|u_{h^{-1},g}-1|_1
    \end{align*}
    for $g\in E$ and $h\in G$.

    Next, by Lemma \ref{lem:approx-1coh-vanish}, we find $W^{(1)}\in \U((M^{\omega}_{\alpha})^K)$ such that, for $V_g^{(2)}=W^{(1)}V_g^{(1)}\alpha_g(W^{(1)\ast})$, we obtain
    \[|V_g^{(2)}\alpha_g(V_h^{(2)})u_{g,h}V_{gh}^{(2)\ast}-1|_1\le4\varepsilon'^{1/2}\]
    and
    \begin{align*}
        |V_p^{(2)}-1|_1 &\le 4\varepsilon^{1/2}+\sup_{h\in\bigcup S}|\alpha_{h^{-1}}^{-1}(V_{h^{-1}}^{(1)})V_p^{(1)}\alpha_p(\alpha_{h^{-1}p}^{-1}(V_{h^{-1}p}^{(1)\ast}))-1|_1\\
        &\le 4\varepsilon^{1/2}+4\varepsilon'^{1/2}+\sup_{h\in\bigcup S}3|u_{h^{-1},p}-1|_1
    \end{align*}
    for $g,h\in E$ and $p\in F$, respectively.
    Choosing $K$-invariant representatives of $V^{(2)}$, equicontinuity on the relevant compact sets allows us to select a common index for which (a) and (b) hold.
\end{proof}

\begin{thm}[2-cohomology vanishing]\label{thm:2coh-vanish}
    Let $(\alpha,u)\colon G\curvearrowright M$ be a centrally free cocycle action of an amenable locally compact group $G$ with a compact open subgroup $K<G$ on a factor.
    \begin{enumerate}
        \item The cocycle $u$ is a coboundary.
        \item Moreover, suppose that $u$ is $K$-invariant.
        Given a finite $K$-$(F,\varepsilon)$-invariant subset $S\subset G/K$, a number $\delta>0$, and a compact subset $\Phi\subset\nS(M)$, one can choose a $K$-invariant $v\colon G\to\U(M)$ so that $1=v_g\alpha_g(v_h)u_{g,h}v_{gh}^{\ast}$ and
        \[\|[v_p,\psi]\|<8\varepsilon^{1/2}+\sup_{h\in\bigcup S}\|[u_{h,h^{-1}}^{\ast}u_{h,h^{-1}p},\psi]\|+\delta\]
        for every $p\in F$ and $\psi\in\Phi$.
    \end{enumerate}
\end{thm}
\begin{proof}
    (1):
    If $M$ is properly infinite, the assertion follows from Proposition \ref{prop:Sutherland}.
    Hence, we assume that $M$ is finite and $u$ is $K$-invariant by \cite[Proposition 4.3]{Ni26}.
    Choose an increasing exhaustion $F_n=F_nK$ of $G$ by compact subsets, and $\varepsilon_n>0$, $\delta_n>0$ such that $\sum_n\varepsilon_n^{1/2}<\infty$ and $\sum_n\delta_n<\infty$.
    For each $n$, let $S_n\subset G/K$ be a finite $K$-$(F_n,\varepsilon_n)$-invariant subset.
    
    Set $\alpha^{(0)}=\alpha$ and $u^{(0)}=u$.
    By Lemma \ref{lem:approx-2coh-vanish}, we inductively obtain $K$-invariant maps $v^{(n)}\colon G\to\U(M)$ such that, if we set $\alpha^{(n+1)}=\Ad{v^{(n)}}\circ\alpha^{(n)}$ and $u_{g,h}^{(n+1)}=v_g^{(n)}\alpha_g^{(n)}(v^{(n)}_h)u_{g,h}^{(n)}v_{gh}^{(n)\ast}$, then
    \begin{enumerate}[label=(\alph*)]
        \item $|u_{g,h}^{(n+1)}-1|_1<\delta_n$,
        \item $|v_p^{(n)}-1|_1<4\varepsilon_n^{1/2}+\sup_{h\in\bigcup S_n}3|u_{h^{-1},p}^{(n)}-1|_1+\delta_n<4\varepsilon_n^{1/2}+3\delta_{n-1}+\delta_n$
    \end{enumerate}
    for all $g,h\in F_{n+1}\cup(\bigcup S_{n+1})^{-1}$ and $p\in F_n$.
    The summability assumption and the condition (b) imply that the sequence $v^{(n)}\cdots v^{(0)}$ converges to a unitary map $v\colon G\to\U(M)$.
    Since $u^{(n+1)}_{g,h}\to 1$ for every $g,h\in G$, the resulting map satisfies $v_g\alpha_g(v_h)u_{g,h}v_{gh}^{\ast}=1$.
    
    (2):
    By (1), there exists a map $\tilde{v}\colon G\to\U(M)$ such that $\tilde{v}_g\alpha_g(\tilde{v}_h)u_{g,h}\tilde{v}_{gh}^{\ast}=1$ for all $g,h\in G$.
    Since $u$ is $K$-invariant, applying \cite[Theorem 4.1]{Ni26} to the 1-cocycle $\tilde{v}|_K$, we may assume $\tilde{v}$ is $K$-invariant.
    Lemma \ref{lem:approx-1coh-vanish} gives $W\in\U((M^K)^{\omega})$ such that
    \[\|[W\tilde{v}_p\alpha_p(W^{\ast}),\psi^{\omega}]\|\le 8\varepsilon^{1/2}+\sup_{h\in\bigcup S}\|[u_{h,h^{-1}}^{\ast}u_{h,h^{-1}p},\psi]\|\]
    for $p\in F$ and $\psi\in\Phi$.
    Take $w\in\U(M^K)$ from a representing sequence for $W$ so that $v_g=w\tilde{v}_g\alpha_g(w^{\ast})$ satisfies the desired properties.
\end{proof}

Let $(G_i)_i$ be a countable family of locally compact groups and $K<G_i$ a common compact open subgroup.
Recall that the amalgamated free product $\bigast_KG_i$ carries a natural locally compact group topology under which $K$ is compact and open (see, for instance, \cite{BP93}).
As an analogue of \cite{Po21}, we prove Theorem \ref{main:2-coh-vanish}.

\begin{proof}[Proof of Theorem \ref{main:2-coh-vanish}]
    Theorem \ref{thm:2coh-vanish} implies that every amenable locally compact group with a compact open subgroup belongs to the class $\VC$.
    Let $(G_i)_i\subset\VC$ be a countable family of groups sharing a common compact open subgroup $K$, and let $(\alpha,u)\colon G=\bigast_KG_i\curvearrowright M$ be a centrally free cocycle action.
    
    We may assume that $u$ is $K$-invariant.
    Each $G_i$ is open in $G$, so the restricted cocycle action on $G_i$ is centrally free.
    Hence, there exists a $K$-invariant map $v^{(i)}\colon G_i\to\U(M)$ such that $v_g^{(i)}\alpha_g(v_h^{(i)})u_{g,h}v_{gh}^{(i)\ast}=1$ for all $g,h\in G_i$.
    Then, for every $i$, the map $G_i\ni g\mapsto v_g^{(i)}\lambda_g^{\alpha,u}\in\U(M\rtimes_{\alpha,u}G)$ is a unitary representation.
    By $K$-invariance of $v^{(i)}$ and the universality, they extend to a unitary representation $G\ni g\mapsto v_g\lambda_g^{\alpha,u}$.
    This completes the proof.
\end{proof}

A basic nonamenable example covered by Theorem \ref{main:2-coh-vanish} is $\SL_2(\mathbb{Q}_p)$, which is an amalgam of two copies of $\SL_2(\mathbb{Z}_p)$.
Further examples concerning totally disconnected groups acting on trees are given in the following corollary.
See \cite[Chapter I, Section 5.4]{Se03}.

\begin{cor}\label{cor:on-tree}
    Let $T$ be a locally finite tree and $G<\Aut(T)$ a closed subgroup acting without inversions.
    Suppose that $G\backslash T$ is a finite tree.
    Then every strictly outer cocycle action of $G$ on $R$ is cocycle conjugate to a genuine action.
\end{cor}
\begin{proof}
    Every vertex stabilizer is compact and open in $G$, and the same holds for every edge stabilizer.
    By the Bass--Serre structure theorem, $G$ is a finite iterated amalgamated free product of its vertex stabilizers over its edge stabilizers.
    The vertex stabilizers are compact and hence amenable, so Theorem \ref{main:2-coh-vanish} shows that $G\in\VC$.
    Since every strictly outer cocycle action on $R$ is centrally free by Proposition \ref{prop:str.outer->cent.free}, the conclusion follows.
\end{proof}

\section{Classification theory}\label{sec:4}

\subsection{Classification of centrally free actions}

In this subsection, we first classify centrally free actions by using Bratteli--Elliott--Evans--Kishimoto intertwining argument following \cite{MT16}.

\begin{lem}\label{lem:classify_step1}
    Let $\alpha,\beta\colon G\curvearrowright M$ be actions of an amenable locally compact group $G$ with a compact open subgroup $K$ on a factor $M$.
    Suppose that $\alpha$ is centrally free and $\alpha|_K=\beta|_K$.
    Fix a faithful normal state $\varphi$ on $M$.
    Let $E=EK\subset G$ be a compact subset, $S\subset G/K$ a finite $K$-$(F,\varepsilon)$-invariant subset, $\delta>0$, and let $D\subset\U(M)$ and $\Phi\subset \nS(M)$ be finite subsets.
    If $\beta_g\circ\alpha_g^{-1}\in\clInt(M)$ for all $g\in G$, then there exist $w\in\U(M^K)$ and a $K$-invariant $\alpha$-cocycle $v\colon G\to\U(M)$ such that
    \begin{enumerate}[label=\textup{(\alph*)}]
        \item $|a(v_p-1)b|_{\varphi}^{\sharp}<4\varepsilon^{1/2}+\delta$;
        \item $\|\Ad(w^{\ast}v_g\alpha_g(w))\circ\alpha_g\circ\beta_g^{-1}(\psi)-\psi\|<\delta$;
        \item $\|[w,\psi]\|<\sup_{s\in S}\|\beta_s\circ\alpha_s^{-1}(\psi)-\psi\|+\delta$
    \end{enumerate}
    for $g\in E$, $p\in F$, $a,b\in D$ and $\psi\in\Phi$.
\end{lem}
\begin{proof}
    We may assume that $K\cup(\bigcup S)\subset E$.
    Choose small $\varepsilon',\delta'>0$ such that $8\varepsilon'^{1/2}+6\delta'<\delta$, and let $S'\subset G/K$ be a finite $K$-$(E,\varepsilon')$-invariant subset.
    Set $\Phi'=\{\alpha_g\circ\beta_g^{-1}(\psi):g\in E,\ \psi\in\Phi\}$, and let $v^{(0)}\colon G\to\U(M)$ be a right $K$-invariant map such that $v^{(0)}_1=1$ and
    \[\|\Ad{v_g^{(0)}}\circ\alpha_h\circ\beta_k\circ\alpha_l^{-1}(\psi')-\beta_g\circ\alpha_{g^{-1}h}\circ\beta_k\circ\alpha_l^{-1}(\psi')\|<\delta'\]
    for $g\in G$, $h\in (\bigcup S')^{-1}$, $k,l\in E$, and $\psi'\in\Phi'$.
    Define a cocycle action $(\tilde{\alpha},u)$ on $M$ by $\tilde{\alpha}=\Ad{v^{(0)}}\circ\alpha$ and $u_{g,h}=v^{(0)}_g\alpha_g(v_h^{(0)})v_{gh}^{(0)\ast}$.
    Then $(\tilde{\alpha},u)$ is $K$-invariant and centrally free.
    We see that $u_{h,h^{-1}}^{\ast}u_{h,h^{-1}g}=\alpha_h(v_{h^{-1}}^{(0)\ast}v_{h^{-1}g}^{(0)})v_g^{(0)\ast}$ and
    \begin{align*}
        \|[u_{h,h^{-1}}^{\ast}u_{h,h^{-1}g},\Ad{v_g^{(0)}}(\psi')]\| &= \|\Ad{\alpha_h(v_{h^{-1}}^{(0)\ast}v_{h^{-1}g}^{(0)})}(\psi')-\Ad{v_g^{(0)}}(\psi')\|\\
        &\le \|\Ad{v_{h^{-1}g}^{(0)}}\circ\alpha_h^{-1}(\psi')-\beta_{h^{-1}g}\circ\alpha_g^{-1}(\psi')\|\\
        &\quad\ +\|\beta_{h^{-1}g}\circ\alpha_g^{-1}(\psi')-\Ad{v_{h^{-1}}^{(0)}}\circ\alpha_h^{-1}\circ\beta_g\circ\alpha_g^{-1}(\psi')\|\\
        &\quad\ +\|\beta_g\circ\alpha_g^{-1}(\psi')-\Ad{v_g^{(0)}}(\psi')\|\\
        &< 3\delta'
    \end{align*}
    for every $g\in E$, $h\in\bigcup gS'$, and $\psi'\in\Phi'$.
    
    Applying Theorem \ref{thm:2coh-vanish} to $(\tilde{\alpha},u)$ and $\Phi''=\{\Ad{v_g^{(0)}}(\psi'):g\in E,\ \psi'\in\Phi'\}$, we obtain a $K$-invariant $v^{(1)}\colon G\to\U(M)$ such that $1=v_g^{(1)}\tilde{\alpha}_g(v_h^{(1)})u_{g,h}v_{gh}^{(1)\ast}$ and
    \[\|[v_g^{(1)},\psi'']\|<8\varepsilon'^{1/2}+\sup_{h\in\bigcup S'}\|[u_{h,h^{-1}}^{\ast}u_{h,h^{-1}g},\psi'']\|+\delta'\]
    for $g\in E$ and $\psi''\in\Phi''$.
    Then $v^{(2)}=v^{(1)}v^{(0)}$ is a $K$-invariant $\alpha$-cocycle.
    For $g\in E$ and $\psi'\in\Phi'$, it follows that $\|[v_g^{(1)},\Ad{v_g^{(0)}}(\psi')]\|<8\varepsilon'^{1/2}+4\delta'$.
    Therefore, we have
    \begin{align*}
        \|[v_g^{(2)},\psi]\| &\le \|[v_g^{(1)},\Ad{v_g^{(0)}}(\psi)]\|+\|\Ad{v_g^{(0)}}(\psi)-\psi\|\\
        &< 8\varepsilon'^{1/2}+5\delta'+\|\beta_g\circ\alpha_g^{-1}(\psi)-\psi\|
    \end{align*}
    and
    \begin{align*}
        \|\Ad{v_g^{(2)}}\circ\alpha_g\circ\beta_g^{-1}(\psi)-\psi\| &\le \|[v_g^{(1)},\psi'']\|+\|\Ad{v_g^{(0)}}\circ\alpha_g\circ\beta_g^{-1}(\psi)-\psi\|\\
        &< 8\varepsilon'^{1/2}+\sup_{h\in\bigcup S'}\|[u_{h,h^{-1}}^{\ast}u_{h,h^{-1}g},\psi'']\|+2\delta'\\
        &< 8\varepsilon'^{1/2}+5\delta'
    \end{align*}
    for every $g\in E$ and $\psi\in\Phi$, where $\psi''=\Ad{v_g^{(0)}}\circ\alpha_g\circ\beta_g^{-1}(\psi)$.

    By Lemma \ref{lem:approx-1coh-vanish}, we can take $w\in\U(M^K)$ so that, if we set $v_g=wv_g^{(2)}\alpha_g(w^{\ast})$, then
    \begin{itemize}
        \item $|a(v_p-1)b|_{\varphi}^{\sharp}<4\varepsilon^{1/2}+\delta'$, which is (a), and
        \item $\|[w,\psi]\|<\sup_{h\in \bigcup S}\|[v_{h^{-1}}^{(2)},\alpha_{h^{-1}}(\psi)]\|+\delta'=\sup_{s\in S}\|[v_s^{(2)},\psi]\|+\delta'$
    \end{itemize}
    for $p\in F$, $a,b\in D$, $\psi\in\Phi$.
    Thus, for all $g\in E$ and $\psi\in\Phi$, it follows that
    \[\|[w,\psi]\|<8\varepsilon'^{1/2}+6\delta'+\sup_{s\in S}\|\beta_s\circ\alpha_s^{-1}(\psi)-\psi\|,\]
    which is (c).
    The preceding estimate implies (b) and we finish the proof.
\end{proof}

\begin{proof}[Proof of Theorem \ref{main:classification}]
    By \cite[Theorem 3.4]{Ni26}, we may assume $\alpha|_K=\beta|_K$.
    Fix a faithful normal state $\varphi$ on $M$.
    Let $F_n=F_nK\subset G$ be a compact subset with $K\subset F_n$ and $\bigcup_nF_n=G$.
    Choose $\varepsilon_n,\delta_n>0$ such that $\sum_n\varepsilon_n^{1/2}<\infty$ and $\sum_n\delta_n<\infty$.
    Let $S_n\subset G/K$ be a finite $K$-$(F_n,\varepsilon_n)$-invariant subset and arrange that $F_n\cup(\bigcup S_n)\subset F_{n+1}$.
    Let $\Phi_n\subset\nS(M)$ be increasing finite subsets with dense union.
    
    Set $\gamma^{(-1)}=\alpha$, $\gamma^{(0)}=\beta$, $\tilde{\Phi}_{-1}=\Phi_0$, and $\tilde{w}^{(-1)}=\tilde{w}^{(-2)}=\tilde{v}^{(-1)}=\tilde{v}^{(-2)}=1$.
    Then, applying Lemma \ref{lem:classify_step1} inductively, there exist $w^{(n)}\in\U(M^K)$ and a $K$-invariant $\gamma^{(n-1)}$-cocycle $v^{(n)}$ such that
    \begin{enumerate}[label=(\alph*)]
        \item $|\tilde{w}^{(n-2)}(v_p^{(n)}-1)\tilde{w}^{(n-2)\ast}\tilde{v}_p^{(n-2)}|_{\varphi}^{\sharp}<4\varepsilon_n^{1/2}+\delta_n$;
        \item $\|\Ad(w^{(n)\ast}v_g^{(n)}\gamma_g^{(n-1)}(w^{(n)}))\circ\gamma_g^{(n-1)}\circ(\gamma_g^{(n)})^{-1}(\psi)-\psi\|<\delta_n$;
        \item $\|[w^{(n)},\psi]\|<\sup_{s\in S_n}\|\gamma_s^{(n)}\circ(\gamma_s^{(n-1)})^{-1}(\psi)-\psi\|+\delta_n$
    \end{enumerate}
    for $g\in F_{n+2}$, $p\in F_n$ and $\psi\in\tilde{\Phi}_n$, where $\tilde{\Phi}_n=\Phi_n\cup\tilde{\Phi}_{n-1}\cup\{\tilde{w}^{(n-1)\ast}\psi\tilde{w}^{(n-1)}:\psi\in\Phi_n\}$.
    Set $\tilde{w}^{(n)}=\tilde{w}^{(n-2)}w^{(n)}$, $\tilde{v}^{(n)}=\tilde{w}^{(n-2)}v^{(n)}\tilde{w}^{(n-2)\ast}\tilde{v}^{(n-2)}$ and $\gamma_g^{(n+1)}=\Ad(w^{(n)\ast}v_g^{(n)}\gamma_g^{(n-1)}(w^{(n)}))\circ\gamma_g^{(n-1)}$.

    For $n\ge 1$ and $\psi\in\Phi_{n-1}$, conditions (b) and (c), together with $\bigcup S_n\subset F_{n+1}$, imply that
    \begin{align*}
        \|\Ad{\tilde{w}^{(n)}}(\psi)-\Ad{\tilde{w}^{(n-2)}}(\psi)\| &= \|[w^{(n)},\psi]\|\\
        &\le \sup_{s\in S_n}\|\gamma_s^{(n)}\circ(\gamma_s^{(n-1)})^{-1}(\psi)-\psi\|+\delta_n\\
        &< \delta_{n-1}+\delta_n,\\
        \|\Ad{\tilde{w}^{(n)\ast}}(\psi)-\Ad{\tilde{w}^{(n-2)\ast}}(\psi)\| &= \|[w^{(n)},\tilde{\psi}]\|\\
        &< \delta_{n-1}+\delta_n,
    \end{align*}
    where $\tilde{\psi}=\tilde{w}^{(n-2)\ast}\psi\tilde{w}^{(n-2)}\in\tilde{\Phi}_{n-1}$, and hence $\Ad{\tilde{w}^{(2n)}}\to\theta_0$, $\Ad{\tilde{w}^{(2n+1)}}\to\theta_1$ in $\Aut(M)$ for some $\theta_0,\theta_1\in\clInt(M)$.
    
    Condition (a) implies $\tilde{v}^{(2n)}$ and $\tilde{v}^{(2n+1)}$ converge pointwise $\ast$-strongly to a $K$-invariant $\alpha$-cocycle $v'$ and a $K$-invariant $\beta$-cocycle $v''$, respectively.
    By construction, 
    \[\gamma^{(2n+1+i)}=\Ad{\tilde{w}^{(2n+i)\ast}}\circ\Ad{\tilde{v}^{(2n+i)}}\circ\gamma^{(i-1)}\circ\Ad{\tilde{w}^{(2n+i)}}\]
    for $i=0,1$.
    Passing to the limit in condition (b) yields $\theta_0^{-1}\circ\Ad{v'}\circ\alpha\circ\theta_0=\theta_1^{-1}\circ\Ad{v''}\circ\beta\circ\theta_1$.
    Moreover, we can choose the cocycle to be $K$-invariant.
\end{proof}

As a corollary, combining with the result of \cite{KST92}, we have the following.
Recall that, for an AFD factor, an automorphism $\alpha$ is approximately inner if and only if its Connes--Takesaki module $\mod(\alpha)$ is trivial.

\begin{cor}
    Let $\alpha,\beta\colon G\curvearrowright M$ be centrally free actions of an amenable locally compact group $G$ with a compact open subgroup $K$ on an AFD factor $M$.
    Then $\alpha$ and $\beta$ are strongly cocycle conjugate if and only if $\mod(\alpha)=\mod(\beta)$.
\end{cor}

Every locally compact group admits a strictly outer action on $R$, e.g., \cite[Corollary 5.2]{Va05}.
Moreover, there is also a simple construction of \emph{minimal} actions on $R$:
Let $G\curvearrowright R$ be faithful, and consider the diagonal action $G\curvearrowright R^{\wtensor\Gamma}\rtimes\Gamma$, where $\Gamma$ is a discrete amenable ICC group and acts on $R^{\wtensor\Gamma}$ by the Bernoulli shift.

\begin{cor}\label{cor:extend-action}
    Let $G$ be a locally compact group, and $H<G$ a closed subgroup.
    Suppose that $H$ is amenable and has a compact open subgroup.
    Then every strictly outer action $\alpha\colon H\curvearrowright R$ admits an $\alpha$-cocycle $v$ such that the action $\Ad{v}\circ\alpha$ extends to a strictly outer action $G\curvearrowright R$.
\end{cor}
\begin{proof}
    Let $\gamma\colon G\curvearrowright R$ be a strictly outer action.
    By Theorem \ref{main:classification}, $\gamma|_H$ and $\alpha$ are cocycle conjugate.
    This means that there exist $\theta\in\Aut(R)$ and an $\alpha$-cocycle $v$ satisfying $\theta\circ\gamma|_H\circ\theta^{-1}=\Ad{v}\circ\alpha$.
    The action $\theta\circ\gamma\circ\theta^{-1}$ is the desired one.
\end{proof}

\begin{rem}
    When $H$ is compact, the cocycle perturbation above may be omitted.
    For noncompact $H$, however, this is not possible in general.
    Consider the case $H=2\mathbb{Z}<\mathbb{Z}=G$, which reduces to the problem of finding an aperiodic automorphism of $R$ which has no square root.
    By \cite[Theorem 7.7]{Co77}, there exists $\beta\in\Aut(R)$ such that $\beta^4=\id$ and $\beta$ has no square root.
    Let $\theta\in\Aut(R)$ be weakly mixing, for instance a Bernoulli shift, and set $\alpha=\beta\otimes\theta\in\Aut(R\wtensor R)$.
    Any square root of $\beta\otimes\theta$ leaves $(R\wtensor R)^{(\beta\otimes\theta)^4}=R\wtensor\mathbb{C}1$ globally invariant, hence it induces a square root of $\beta\in\Aut(R\wtensor\mathbb{C}1)$, a contradiction.
\end{rem}

\subsection{McDuff-type absorption}

We obtain the following characterization of central freeness of actions on McDuff factors.
We use the dynamical McDuff-type theorem from \cite{SW24}.
The uniqueness theorem also implies that every strictly outer action of an amenable group with a compact open subgroup on $R$ is strongly self-absorbing.

\begin{cor}\label{cor:cent.free-McDuff}
    Let $\beta\colon G\curvearrowright M$ be an action of an amenable locally compact group $G$ with a compact open subgroup on a McDuff factor $M$.
    Then the following conditions are equivalent.
    \begin{enumerate}
        \item The action $\beta$ is centrally free;
        \item The action $\beta$ is cocycle conjugate to $\beta\otimes\alpha$ for every action $\alpha\colon G\curvearrowright R$;
        \item The action $\beta$ is cocycle conjugate to $\beta\otimes\gamma$ for a strictly outer action $\gamma\colon G\curvearrowright R$;
        \item There exists an equivariant embedding $R\hookrightarrow M_{\omega,\beta}$ for every action $G\curvearrowright R$;
        \item There exists an equivariant embedding $R\hookrightarrow M_{\omega,\beta}$ for a strictly outer action $G\curvearrowright R$.
    \end{enumerate}
\end{cor}
\begin{proof}
    First, the implications (2)$\implies$(3) and (4)$\implies$(5)$\implies$(1) are immediate.

    Suppose that (1) holds, and let $\alpha\colon G\curvearrowright R$ be an action.
    There exists a McDuff isomorphism $M\cong M\wtensor R$ that is approximately unitarily equivalent to the canonical embedding $M\ni x\mapsto x\otimes 1\in M\wtensor R$.
    The action $\beta\otimes\alpha$ is centrally free and, under this identification, differs pointwise from $\beta$ by approximately inner automorphisms.
    Theorem \ref{main:classification} therefore proves (2).

    Suppose that (3) holds, and let $\alpha\colon G\curvearrowright R$ be an action.
    Since $\gamma\otimes\alpha^{\otimes\infty}\colon G\curvearrowright R\wtensor R^{\wtensor\infty}\cong R$ is strictly outer, $\gamma$ and $\gamma\otimes\alpha^{\otimes\infty}$ are cocycle conjugate.
    Together with the assumption (3), this shows that $\beta$ and $\beta\otimes\alpha^{\otimes\infty}$ are cocycle conjugate.
    The sequence of embeddings of ($\alpha\colon G\curvearrowright R$) into the successive tensor factors of $\beta\otimes\alpha^{\otimes\infty}$ yields (4).
\end{proof}

The proof also shows that the cocycle conjugacy may be replaced by the strong cocycle conjugacy in the statement of the corollary.

\begin{cor}
    Let $G\curvearrowright M$ be a centrally free action of an amenable locally compact group $G$ with a compact open subgroup on a McDuff factor $M$.
    \begin{enumerate}
        \item The inclusion $(M\subset M\rtimes G)$ is strongly stable, i.e., isomorphic to $(M\wtensor R\subset (M\rtimes G)\wtensor R)$.
        \item The inclusion $(M^Kp_K\subset p_K(M\rtimes G)p_K)$ is strongly stable.
    \end{enumerate}
\end{cor}
\begin{proof}
    By Corollary \ref{cor:cent.free-McDuff}, the action absorbs the trivial action on $R$.
\end{proof}

Let $(N\subset M)$ be an irreducible inclusion of $\II_1$ factors.
The \emph{relative fundamental group} $\mathcal{F}(N\subset M)$ is the subgroup of $\mathbb{R}_+^{\ast}$ consisting of those $t>0$ for which $(N\subset M)$ is isomorphic to the amplified inclusion $(N^t\subset M^t)$.
Recall that $M\rtimes G$ is a semifinite factor if $G$ is unimodular, $M$ is finite, and $G\curvearrowright M$ is strictly outer.
In this case, $p_K\in M\rtimes G$ is a finite projection.

\begin{cor}
    Let $G\curvearrowright M$ be a centrally free action of an amenable locally compact unimodular group $G$ with a compact open subgroup on a McDuff factor of type $\II_1$.
    Then
    \[\mathcal{F}(M^Kp_K\subset p_K(M\rtimes G)p_K)=\mathbb{R}_+^{\ast}.\]
\end{cor}
\begin{proof}
    By the preceding corollary, the inclusion is isomorphic to its tensor product with $R$.
    For every $t>0$, amplify on the $R$-tensor factor and use $R^t\cong R$.
    It then follows that the amplified inclusion is isomorphic to the original one.
\end{proof}

For non centrally free actions, we have the following partial analogue of the classification theorem.
It may also be regarded as a locally compact analogue of \cite[Theorem 3.11(ii)]{SW24}.
We use the relative form of the results from Section \ref{sec:2} for the faithful quotient action on $M_{\omega,\alpha}$.

\begin{thm}
    Let $\alpha\colon G\curvearrowright M$ be a strictly outer action of a locally compact group $G$ on a McDuff factor $M$.
    Let $H=\ker(G\curvearrowright M_{\omega,\alpha})<G$ and $\gamma=\gamma^{G/H}\colon G\curvearrowright R$ denote the inflation of a strictly outer action $G/H\curvearrowright R$.
    Suppose that $G/H$ is amenable and has a compact open subgroup.
    Then $\alpha$ is cocycle conjugate to $\alpha\otimes\gamma$ if and only if $M_{\omega,\alpha}$ contains a copy of $R$.
\end{thm}
\begin{proof}
    The subgroup $H$ is closed and normal.
    Put $\dot{G}=G/H$ and take a compact open subgroup $\dot{K}<\dot{G}$.
    Then the induced action $\dot{\alpha}\colon\dot{G}\curvearrowright M_{\omega,\alpha}$ is faithful.
    By \cite[Theorem 5.4]{SW24}, the action $\alpha$ is cocycle conjugate to $\alpha\otimes\gamma$ if and only if there exists a $\gamma$-$\alpha$-equivariant embedding $R\hookrightarrow M_{\omega,\alpha}$.
    Thus, it remains to show that a given embedding $\theta_1\colon R\hookrightarrow M_{\omega,\alpha}$ can be replaced by an equivariant one.
    
    Set $Q=\bigvee_{g\in G}\alpha_g(\theta_1(R))\subset M_{\omega,\alpha}$.
    As in the proof of Proposition \ref{prop:cpt-cent.free->Rohlin}, we have a $\dot{K}$-equivariant embedding $L^{\infty}(\dot{K})\hookrightarrow M_{\omega,\alpha}$ whose image is orthogonal to $Q$.
    Define $\theta_2\colon R\hookrightarrow L^{\infty}(\dot{K})\wtensor Q\subset M_{\omega,\alpha}$ by $[\theta_2(x)](\dot{k})=\dot{\alpha}_{\dot{k}}(\theta_1(\gamma_{k^{-1}}(x)))$ for $x\in R$.
    Then we have
    \[[\dot{\alpha}_{\dot{k}}(\theta_2(x))](\dot{l})=\dot{\alpha}_{\dot{k}}([\theta_2(x)](\dot{k}^{-1}\dot{l}))=\dot{\alpha}_{\dot{l}}(\theta_1(\gamma_{l^{-1}k}(x)))=[\theta_2(\gamma_k(x))](\dot{l})\]
    and hence $\theta_2$ is $\dot{K}$-equivariant.
    Consequently, the map $\dot{G}\ni \dot{g}\mapsto\dot{\alpha}_{\dot{g}}\circ\theta_2\circ\gamma_{g^{-1}}$ is right $\dot{K}$-invariant.

    Let $\dot{F}=\dot{F}\dot{K}\subset\dot{G}$ be compact, $\varepsilon>0$ and $\dot{S}\subset\dot{G}/\dot{K}$ a finite $\dot{K}$-$(\dot{F},\varepsilon)$-invariant subset.
    Choose a Rohlin tower $(e_{\dot{s}})_{\dot{s}\in\dot{S}}$ orthogonal to $\bigvee_{g\in G}\alpha_g(\theta_2(R))$, and define $\theta_{\dot{S}}\colon R\hookrightarrow M_{\omega,\alpha}$ by $\theta_{\dot{S}}(x)=\sum_{\dot{s}}\dot{\alpha}_{\dot{s}}(\theta_2(\gamma_{s^{-1}}(x)))e_{\dot{s}}$.
    This is approximately equivariant on $\dot{F}$, and the usual index selection trick guarantees the desired equivariant embedding.
\end{proof}

By Remark \ref{rem:need-minimality}, the condition $R\hookrightarrow M_{\omega,\alpha}$ may fail even when the action is strictly outer.
Indeed, take $\gamma\colon\mathbb{T}\curvearrowright R$ from Remark \ref{rem:need-minimality} and let $\beta\colon \mathbb{T}\curvearrowright M$ be a minimal action on a full factor.
Then $\beta\otimes\gamma\colon \mathbb{T}\curvearrowright M\wtensor R$ is strictly outer on the McDuff factor, whereas $(M\wtensor R)_{\omega,\beta\otimes\gamma}=\mathbb{C}1\wtensor R_{\omega,\gamma}$.
The algebra $R_{\omega,\gamma}$ in this example is abelian and hence contains no copy of $R$.

\section{Crossed product inclusions}\label{sec:5}

\subsection{Relative automorphisms}

In this section, we describe how the group structure of $G$ is encoded by the crossed product inclusion.
We begin with a minor extension of \cite[Corollary 3.11]{BB18} to cocycle actions.

\begin{prop}\label{prop:Weyl-group}
    Let $(\alpha,u)\colon G\curvearrowright M$ be a strictly outer \emph{cocycle} action of a locally compact group on a factor.
    Then $\N_{M\rtimes_{\alpha,u}G}(M)=\{a\lambda_g^{\alpha,u}:a\in\U(M),\ g\in G\}$.
\end{prop}
\begin{proof}
    Consider the cocycle action $(\alpha\otimes\id,u\otimes 1)$ on $M\wtensor\mathbb{B}$.
    This cocycle action is perturbed to a strictly outer action $\gamma$, and we have
    \[(M\rtimes_{\alpha,u}G)\wtensor\mathbb{B}=(M\wtensor\mathbb{B})\rtimes_{\alpha\otimes\id,u\otimes1}G=(M\wtensor\mathbb{B})\rtimes_{\gamma}G.\]
    Let $w\in\N_{M\rtimes_{\alpha,u}G}(M)$.
    The unitary $w\otimes1\in\U((M\rtimes_{\alpha,u}G)\wtensor\mathbb{B})$ normalizes $M\wtensor\mathbb{B}$.
    By \cite[Corollary 3.11]{BB18}, there exist $b\in\U(M\wtensor\mathbb{B})$ and $g\in G$ such that $w\otimes 1=b\lambda_g^{\gamma}$.
    Now the unitary $\lambda_g^{\gamma}$ has the form $c(\lambda_g^{\alpha,u}\otimes1)$ for some $c\in\U(M\wtensor\mathbb{B})$.
    Therefore, $(w(\lambda_g^{\alpha,u})^{\ast})\otimes 1=bc\in M\wtensor\mathbb{B}$, which means $a=w(\lambda_g^{\alpha,u})^{\ast}\in\U(M)$.
\end{proof}

\begin{lem}\label{lem:N/U}
    Let $G\curvearrowright M$ be a strictly outer cocycle action of a locally compact group $G$ on a factor $M$.
    Then there exists a canonical isomorphism $\N_{M\rtimes G}(M)/\U(M)\cong G$ of topological groups.
\end{lem}
\begin{proof}
    The Borel map $G\ni g\mapsto \lambda_g\U(M)\in\N_{M\rtimes G}(M)/\U(M)$ is an algebraic group isomorphism.
    Both groups $G$ and $\N_{M\rtimes G}(M)/\U(M)$ are Polish, so the inverse map is also Borel.
    Since these maps are Borel homomorphisms between Polish groups, they are continuous.
\end{proof}

Consequently, the cocycle action $G\curvearrowright M$ is cocycle conjugate to an action (the 2-cocycle is a coboundary) if and only if the following short exact sequence of topological groups splits:
\[1\longrightarrow\U(M)\longrightarrow\N_{M\rtimes G}(M)\longrightarrow G\longrightarrow1.\]

\begin{proof}[Proof of Theorem \ref{main:short-exact}]
    Lemma \ref{lem:N/U} provides a canonical homomorphism $\Aut(R\subset R\rtimes G)\to\Aut(\N_{R\rtimes G}(R)/\U(R))=\Aut(G)$.
    To prove surjectivity, let $\kappa\in\Aut(G)$.
    Since the actions $\alpha$ and $\alpha\circ\kappa$ are cocycle conjugate, there exist $\theta\in\Aut(R)$ and an $\alpha$-cocycle $v\colon G\to\U(R)$ such that $\theta\circ\Ad{v}\circ\alpha\circ\theta^{-1}=\alpha\circ\kappa$.
    Define $\tilde{\kappa}\in\Aut(R\subset R\rtimes G)$ by
    \[\tilde{\kappa}(x\lambda_g)=\theta(xv_g^{\ast})\lambda_{\kappa(g)}\]
    for $x\in R$ and $g\in G$.
    This automorphism $\tilde{\kappa}$ induces $\kappa\in\Aut(G)$.
    Therefore the canonical homomorphism $\Aut(R\subset R\rtimes G)\to\Aut(G)$ is surjective.
    
    Let $\theta$ be an automorphism in the kernel of $\Aut(R\subset R\rtimes G)\to\Aut(G)$.
    Then the map $v\colon G\to\U(R)$ given by $v_g=\theta(\lambda_g)\lambda_g^{\ast}$ is an $\alpha$-cocycle and satisfies $\theta|_R\circ\alpha\circ\theta|_R^{-1}=\Ad{v}\circ\alpha$.
    By perturbing $\theta$ by an element of $\Int(R\subset R\rtimes G)$, we may assume that $v$ is $K$-invariant.
    It follows that $\theta(R^K)=R^K$ and $\theta|_R\circ\alpha_k=\alpha_k\circ\theta|_R$ for every $k\in K$. 
    Since $\Aut(R)=\clInt(R)$, there exists a sequence of unitaries $w_n\in\U(R)$ such that $\Ad(w_n)\to\theta|_R$.
    Using $L^{\infty}(K)\wtensor R\subset R^{\omega}_{\alpha}$, take $W_n\in\U(L^{\infty}(K)\wtensor R)\subset\U(R^{\omega}_{\alpha})$ so that $W_n(k)=\alpha_k(w_n)$.
    Then we obtain that $W_n\in\U((R^K)^{\omega})$ and $\Ad{W_n}|_R\to\theta|_R$ because
    \[[W_nxW_n^{\ast}](k)=\alpha_k(w_n)x\alpha_k(w_n^{\ast})\to\alpha_k\circ\theta|_R\circ\alpha_k^{-1}(x)=\theta(x)\]
    for $x\in R$ and uniformly for $k\in K$.
    By choosing $w_n$ from a representing sequence of $W_n$, we assume that $w_n\in\U(R^K)$ and $\Ad{w_n}\to\theta|_R$.

    Let $W=(w_n)^{\omega}\in\U((R^K)^{\omega})$.
    For all $g\in G$, we have
    \[\Ad[W^{\ast}v_g\alpha_g(W)]|_R=\theta|_R^{-1}\circ\Ad{v_g}\circ\alpha_g\circ\theta|_R\circ\alpha_g^{-1}=\id.\]
    Thus, $V_g=W^{\ast}v_g\alpha_g(W)$ belongs to $\U(R_{\omega,\alpha})$, and $V\colon G\to\U(R_{\omega,\alpha})$ is an $\alpha$-cocycle.
    By Theorem \ref{thm:1-coh-vanish}, there exists a unitary $W'\in\U((R_{\omega,\alpha})^K)$ such that $V_g=W'\alpha_g(W'^{\ast})$.
    It follows that $\Ad{WW'}|_R=\theta|_R$ and $WW'\alpha_g((WW')^{\ast})=v_g$.
    This means that $WW'$ implements $\theta$ on $R\rtimes G$ in the ultrapower.
    Here, we used the identification $WW'\in R^{\omega}_{\alpha}\subset(R\rtimes G)^{\omega}$ by \cite{To17}.
    Taking a suitable representing sequence of $WW'$ shows that $\theta\in\clInt(R\subset R\rtimes G)$.
    The reverse inclusion in the kernel is clear, which completes the proof.
\end{proof}

For $\theta\in\Aut(R\subset R\rtimes G)$, let $\kappa_{\theta}\in\Aut(G)$ denote the induced automorphism.
By construction, we see that $\theta(R\rtimes H)=R\rtimes\kappa_{\theta}(H)$ for every closed subgroup $H<G$.
Thus, the Galois correspondence (Theorem \ref{thm:ISP}) intertwines the natural actions on the lattice of closed subgroups of $G$ and on the lattice of intermediate subfactors.

\begin{cor}% [Naturality of the Galois correspondence]
    Let $G$ be an amenable totally disconnected locally compact group, and $G\curvearrowright R$ a strictly outer action.
    For a locally compact group $H$, the following conditions are equivalent.
    \begin{enumerate}
        \item There exists a topological group embedding $\kappa\colon H\hookrightarrow G$;
        \item For a strictly outer action $H\curvearrowright R$, there exists an embedding $\theta\colon R\rtimes H\hookrightarrow R\rtimes G$ such that $\theta(R)=R$.
    \end{enumerate}
\end{cor}
\begin{proof}
    Suppose that (1) holds and identify $H$ with the closed subgroup $\kappa(H)<G$.
    The condition (2) follows by combining Corollary \ref{cor:extend-action} with the uniqueness of strictly outer actions of $G$ on $R$.

    Conversely, we suppose that (2) holds.
    Since $\theta(R\rtimes H)$ is an intermediate subfactor of $(R\subset R\rtimes G)$, Theorem \ref{thm:ISP} implies that there exists a closed subgroup $H'<G$ with $\theta(R\rtimes H)=R\rtimes H'\subset R\rtimes G$.
    By Lemma \ref{lem:N/U}, it follows that $H\cong H'<G$.
\end{proof}

\subsection{Characterization of crossed products}

In this subsection, we characterize crossed-product inclusions arising from strictly outer actions in terms of regularity, irreducibility, and the topology of the normalizer quotient.
We first treat compact normalizer quotients and then pass to the general case through an expected intermediate factor.
This is viewed as a locally compact analogue of Choda's characterization of crossed-product factorizations (\cite{Ch79}).

\begin{lem}\label{lem:cpt-factrization}
    Let $(M\subset M_1)$ be a regular and irreducible inclusion of factors with separable predual.
    If $K=\N_{M_1}(M)/\U(M)$ is compact, then there exists a strictly outer cocycle action $(\alpha,u)\colon K\curvearrowright M$ such that $M_1=M\rtimes_{\alpha,u}K$.
\end{lem}
\begin{proof}
    Let $v\colon K\to\N_{M_1}(M)$ be a Borel lift and set $\alpha_k=\Ad{v_k}|_M\in\Aut(M)$, $u_{k,l}=v_kv_lv_{kl}^{\ast}\in\U(M)$.
    Then $(\alpha,u)$ is a cocycle action of $K$ on $M$.
    Moreover, it is outer by irreducibility.

    We claim that there exists a normal homomorphism $M\rtimes_{\alpha,u}K\to M_1$ such that $x\lambda^{\alpha,u}_k\mapsto xv_k$ for $x\in M$ and $k\in K$.
    Indeed, $M\rtimes_{\alpha,u}K$ is represented on the Hilbert space $L^2(M_1)\otimes L^2(K)$ by $x\mapsto x\otimes1$ and $\lambda^{\alpha,u}_k\mapsto v_k\otimes\lambda_k$.
    Let $p_K$ be the projection of $L^2(K)$ onto the constant functions.
    Since $1\otimes p_K$ commutes with this representation of $M\rtimes_{\alpha,u}K$, compression by $1\otimes p_K$ gives the homomorphism.
    Surjectivity follows from the regularity of the inclusion.
    
    Let $z\in\Z(M\rtimes_{\alpha,u}K)$ be the support projection of this homomorphism.
    The irreducibility of $(M\subset M_1)$ implies that $z$ is a minimal projection in $M'\cap (M\rtimes_{\alpha,u}K)$.
    After tensoring with $\mathbb{B}$, the cocycle action $(\alpha\otimes\id,u\otimes 1)$ on $M\wtensor\mathbb{B}$ is cocycle conjugate to an action $\gamma\colon K\curvearrowright M\wtensor\mathbb{B}$.
    Note that $\gamma$ is outer since $\alpha$ is outer.
    We identify
    \[(M\rtimes_{\alpha,u}K)\wtensor\mathbb{B}=(M\wtensor\mathbb{B})\rtimes_{\gamma}K\]
    and
    \[(M\wtensor\mathbb{B})'\cap((M\wtensor\mathbb{B})\rtimes_{\gamma}K)=(M'\cap (M\rtimes_{\alpha,u}K))\wtensor\mathbb{C}1.\]
    The projection $z\otimes 1$ in that relative commutant is minimal, and hence the relative commutant is not diffuse.
    By \cite[Lemma 4.3 and Theorem 4.1]{BMO20}, the action $\gamma$ is minimal and $z=1$.
    Therefore we have $M_1=M\rtimes_{\alpha,u}K$ and $(\alpha,u)$ is strictly outer.
\end{proof}

\begin{lem}[cf.\ {\cite[Theorem 5.3]{BB18}}]\label{lem:regular-irreducible}
    Let $M\subset M_1\subset M_2$ be factors.
    Suppose that $(M\subset M_2)$ is irreducible and regular, and $(M_1\subset M_2)$ has expectation.
    \begin{enumerate}
        \item If $M_2=M\rtimes G$ for a strictly outer cocycle action $G\curvearrowright M$ of a locally compact group $G$, then there exists an open subgroup $U<G$ such that $M_1=M\rtimes U$.
        \item If $M_1=M\rtimes U$ for a strictly outer cocycle action $U\curvearrowright M$ of a locally compact group $U$, then there exists a group $G$ containing $U$ as an open subgroup, and an extension of the cocycle action to $G\curvearrowright M$ such that $M_2=M\rtimes G$.
    \end{enumerate}
\end{lem}
\begin{proof}
    (1):
    This is \cite[Theorem 5.3]{BB18}.
    Its proof applies without change to cocycle actions.

    (2):
    Let $\mathbb{E}\colon M_2\to M_1$ be a faithful normal conditional expectation.
    For $v\in\N_{M_2}(M)$, we see that $v^{\ast}\mathbb{E}(v)\in M'\cap M_2=\mathbb{C}1$ and hence
    \[\mathbb{E}(v)=\begin{cases}
        v&\text{if }v\in\N_{M_1}(M),\\
        0&\text{if }v\notin\N_{M_1}(M).
    \end{cases}\]
    It follows that $\N_{M_1}(M)$ is open in $\N_{M_2}(M)$, which implies that $U=\N_{M_1}(M)/\U(M)$ is an open subgroup of $G=\N_{M_2}(M)/\U(M)$.
    In particular, $G$ is locally compact.

    Choose a section $\sigma\colon G/U\to G$ with $\sigma(U)=1_G$ and a Borel lift $v\colon G\to\N_{M_2}(M)$ satisfying $v_g=v_{\sigma(gU)}\lambda_{\sigma(gU)^{-1}g}$.
    This defines a cocycle action of $G$ on $M$ extending the original cocycle action of $U$.
    We obtain $\mathbb{E}(v_{\sigma(s)}^{\ast}v_{\sigma(s')})=0$ whenever $s\neq s'$.
    There exist canonical orthogonal decompositions
    \[L^2(M\rtimes G)=\bigoplus_{s\in G/U}\lambda_{\sigma(s)}L^2(M\rtimes U),\quad L^2(M_2)=\bigoplus_{s\in G/U}v_{\sigma(s)}L^2(M_1).\]
    Since $M_1=M\rtimes U$, the unitary sending $\lambda_{\sigma(s)}\xi$ to $v_{\sigma(s)}\xi$ intertwines the two representations.
    Thus, the correspondence $x\lambda_g\mapsto xv_g$ provides an isomorphism $M\rtimes G\cong M_2$.
    Finally, strict outerness follows from the irreducibility of $(M\subset M_2)$.
\end{proof}

\begin{proof}[Proof of Theorem \ref{main:factrization}]
    This is a direct consequence of Lemmas \ref{lem:cpt-factrization} and \ref{lem:regular-irreducible}.
    Take a Borel lift $K\to\N_{q^{-1}(K)''}(M)\subset\N_{M_1}(M)$ and the same argument as in the proof of Lemma \ref{lem:cpt-factrization} shows the crossed product factorization $q^{-1}(K)''=M\rtimes K$.
    Now Lemma \ref{lem:regular-irreducible} applies.
\end{proof}

\begin{exam}
    Let $\alpha\colon G\curvearrowright R$ be a strictly outer action of a locally compact group $G$ with a compact open subgroup $K$.
    By the modular formula for the dual weight and the factoriality of its centralizer $R\rtimes\ker{\Delta_G}$, the S-invariant of $R\rtimes G$ is given by $\mathrm{S}(R\rtimes G)=\overline{\Delta_G(G)}$.
    For $g\in G$, comparing the size of $K$ and $g^{-1}Kg$, we have
    \[\Delta_G(g)=\frac{|g^{-1}Kg|}{|K|}=\frac{[g^{-1}Kg:K\cap g^{-1}Kg]}{[K:K\cap g^{-1}Kg]}\in\mathbb{Q}.\]
    It follows that if $R\rtimes G$ is of type $\III_{\lambda}$ for $\lambda\in(0,1)$, then $\lambda\in(0,1)\cap\mathbb{Q}$.

    Conversely, every $\lambda\in(0,1)\cap\mathbb{Q}$ occurs in this way.
    Consider the amenable, totally disconnected group $G_{\lambda}=\mathbb{A}_{\mathbb{Q},\mathrm{f}}\rtimes_{\lambda}\mathbb{Z}$, where $\mathbb{A}_{\mathbb{Q},\mathrm{f}}$ is the finite adele ring and $\mathbb{Z}$ acts by multiplication by $\lambda$.
    This group satisfies $\Delta_{G_{\lambda}}(G_{\lambda})=\lambda^{\mathbb{Z}}$ and hence $R\rtimes G_{\lambda}$ is isomorphic to the hyperfinite factor of type $\III_{\lambda}$ for every strictly outer action $G_{\lambda}\curvearrowright R$.
\end{exam}

\begin{rem}
    Fix an irrational number $\lambda\in(0,1)\setminus\mathbb{Q}$ and put $\Gamma=\bigoplus_{\mathbb{N}}\mathbb{Z}/2\mathbb{Z}$.
    Let $\beta\colon\Gamma\curvearrowright R$ be an outer action.
    On $X=\{0,1\}^{\mathbb{N}}$ equipped with the product measure $\mu_{\lambda}=(\frac1{1+\lambda},\frac\lambda{1+\lambda})^{\mathbb{N}}$, let $\rho\colon\Gamma\curvearrowright X$ be the action by coordinatewise addition.
    We can see that the inclusion $R\cong R\rtimes_\beta\Gamma\subset(R\wtensor L^\infty(X))\rtimes_{\beta\otimes\rho}\Gamma\cong R_{\lambda}$ is regular and irreducible.
    Consequently, there exists a regular irreducible inclusion $(R\subset R_\lambda)$ that is not isomorphic to $(R\subset R\rtimes G)$ for any locally compact group $G$ with a compact open subgroup.
\end{rem}

\begin{cor}
    Let $G$ be an amenable locally compact group with a compact open subgroup $K$.
    Then the inclusion of von Neumann algebras $(R\subset Q)$, where $R$ is the hyperfinite $\II_1$ factor, is isomorphic to $(R\subset R\rtimes G)$ for some (hence every) strictly outer action $G\curvearrowright R$ if and only if the following conditions hold.
    \begin{enumerate}
        \item The inclusion $(R\subset Q)$ is regular and irreducible.
        \item There exists a topological group isomorphism $\N_Q(R)/\U(R)\cong G$.
        \item The inclusion $(q^{-1}(K)''\subset Q)$ has expectation, where $q\colon\N_Q(R)\to\N_Q(R)/\U(R)=G$ denotes the canonical quotient.
    \end{enumerate}
\end{cor}

\enlargethispage{\baselineskip}

\end{document}